\documentclass[a4paper]{article}

\usepackage{etoolbox}
\usepackage{expl3}
\usepackage{xifthen}
\usepackage{xparse}

\usepackage[utf8]{inputenc}
\usepackage[T1]{fontenc}
\usepackage{geometry}
\usepackage{authblk} 

\usepackage[british]{babel}
\usepackage{enumitem} 
\usepackage{xcolor} 
\usepackage{lineno}  
\usepackage{lmodern}  
\usepackage[defaultlines=2]{nowidow}  
\usepackage[hyphens]{url} 
\usepackage{hyperref} 
\usepackage{microtype}  
\usepackage{orcidlink} 
\usepackage{csquotes} 
\hypersetup{
	colorlinks = true,  
	linkcolor = stdred,  
	citecolor = stdgreen,  
	urlcolor = stdblue,  
}

\newlength\oldparindent
\renewcommand{\subparagraph}[1]{\medskip\noindent\textbf{#1}}

\setlist{noitemsep,labelwidth=*,leftmargin=*,align=left}
\setlist[1]{labelindent=1em}
\setlist[enumerate,1]{label=(\alph*)}
\setlist[enumerate,2]{label=(\arabic*),ref=\theenumi~(\arabic*)}
\setlist[description]{font=\normalfont,leftmargin=!}

\newlist{enumprop}{enumerate}{2}  
\setlist[enumprop,1]{label=(\alph*),ref=(\alph*)}
\setlist[enumprop,2]{label=(\arabic*),ref=\theenumpropi~(\arabic*)}
\newlist{enumcond}{enumerate}{2}  
\setlist[enumcond,1]{label=(\roman*),ref=(\roman*)}
\setlist[enumcond,2]{label=(\arabic*),ref=\theenumcondi~(\arabic*)}
\newlist{enumstep}{enumerate}{2}  
\setlist[enumstep,1]{label=(\arabic*),ref=(\arabic*)}
\setlist[enumstep,2]{label=(\alph*),ref=\theenumstepi~(\alph*)}
\newlist{enumpart}{enumerate}{2}  
\setlist[enumpart,1]{label=(\alph*),ref=(\alph*)}
\setlist[enumpart,2]{label=(\arabic*),ref=\theenumparti~(\arabic*)}
\newlist{enumop}{enumerate}{2}  
\setlist[enumop,1]{label=(\roman*),ref=(\roman*)}
\setlist[enumop,2]{label=(\arabic*),ref=\theenumopi~(\arabic*)}
\newlist{enummod}{enumerate}{2}  
\setlist[enummod,1]{label=(\arabic*),ref=(\arabic*)}
\setlist[enummod,2]{label=(\alph*),ref=\theenummodi~(\alph*)}

\usepackage{graphicx}
\usepackage{caption}
\usepackage{subcaption}
\usepackage{rotating}
\usepackage{multirow} 
\definecolor{stdred}{RGB}{255,31,91}
\definecolor{stdgreen}{RGB}{0,205,108}
\definecolor{stdblue}{RGB}{0,152,222}
\definecolor{stdpurple}{RGB}{175,88,186}
\definecolor{stdyellow}{RGB}{255,198,30}
\definecolor{stdorange}{RGB}{242,133,34}
\definecolor{stdbrown}{RGB}{166,118,29}
\definecolor{stdgrey}{RGB}{160,177,186} 
\definecolor{grey}{RGB}{127,127,127} 
\usepackage{pdfpages}
\usepackage{array}

\newcolumntype{L}[2]{>{#1\raggedright\let\newline\\\arraybackslash\hspace{0pt}}m{#2}}
\newcolumntype{C}[2]{>{#1\centering\let\newline\\\arraybackslash\hspace{0pt}}m{#2}}
\newcolumntype{R}[2]{>{#1\raggedleft\let\newline\\\arraybackslash\hspace{0pt}}m{#2}}

\usepackage{amsmath}
\usepackage{framed}
\usepackage[framed, hyperref, thmmarks, amsmath]{ntheorem}
\usepackage{mathrsfs} 

\usepackage{amsfonts}
\usepackage{amssymb}
\usepackage{mathtools}
\usepackage{extarrows}
\usepackage{old-arrows} 
\usepackage{colonequals}

\theoremstyle{change}
\theoremheaderfont{\normalfont\bfseries}
\theorembodyfont{\normalfont}
\theorempreskip{\topsep}
\theorempostskip{\topsep}
\theoremseparator{.}
\theoremsymbol{\ensuremath{\triangleleft}}
\newtheorem{definition}{Definition}[section]

\newtheorem{lemma}[definition]{Lemma}
\newtheorem{example}[definition]{Example}
\newtheorem{theorem}[definition]{Theorem}
\newtheorem{corollary}[definition]{Corollary}

\newtheorem{notation}[definition]{Notation}

\newtheorem{observation}[definition]{Observation}

\theoremstyle{plain}

\renewtheorem*{claim*}{Claim}
\makeatletter
\@addtoreset{claim}{definition} 
\makeatother

\theoremstyle{changebreak}
\theoremheaderfont{\normalfont\bfseries}
\theorembodyfont{\normalfont}
\theorempreskip{\topsep}
\theorempostskip{\topsep}
\theoremseparator{.}
\theoremsymbol{\ensuremath{\triangleleft}}

\theoremstyle{break}

\makeatletter
\@addtoreset{claim}{definition} 
\makeatother

\theoremstyle{change}
\theoremheaderfont{\normalfont\bfseries}
\theorembodyfont{\normalfont\slshape}
\theorempreskip{\topsep}
\theorempostskip{\topsep}
\theoremseparator{.}
\theoremsymbol{\ensuremath{\triangleleft}}
\newtheorem{lemmacite}[definition]{Lemma}

\theoremstyle{nonumberplain}
\theoremheaderfont{\itshape}
\theorembodyfont{\upshape}
\theorempreskip{\topsep}
\theorempostskip{\topsep}
\theoremseparator{.}
\theoremsymbol{}

\theoremsymbol{\ensuremath{\square}}
\newtheorem{proof}{Proof}
\theoremsymbol{\ensuremath{\diamond}}

\theoremstyle{nonumberbreak}
\theoremheaderfont{\itshape}
\theorembodyfont{\upshape}
\theorempreskip{\topsep}
\theorempostskip{\topsep}
\theoremseparator{.}
\theoremsymbol{}

\theoremsymbol{\ensuremath{\square}}

\theoremsymbol{\ensuremath{\diamond}}

\ExplSyntaxOn
\NewDocumentCommand{\instring}{mmmm}
{
	\oleks_instring:nnnn { #1 } { #2 } { #3 } { #4 }
}
\tl_new:N \l__oleks_instring_test_tl
\cs_new_protected:Nn \oleks_instring:nnnn
{
	\tl_set:Nn \l__oleks_instring_test_tl { #1 }
	\regex_match:nnTF { \u{l__oleks_instring_test_tl} } { #2 } { #3 } { #4 }
}
\ExplSyntaxOff

\newcommand{\RR}{\ensuremath{\mathbb{R}}}

\def\setdelimiter{\colon\,}
\def\replacecolonbydelimiter#1:#2\relax{#1\setdelimiter#2}
\def\setcontents#1{\replacecolonbydelimiter#1\relax}
\NewDocumentCommand{\Set}{m}{
	\instring{:}{#1}{
		\ensuremath{\left\{\setcontents{#1}\right\}}%
	}{
		\ensuremath{\left\{#1\right\}}%
	}%
}

\renewcommand{\emptyset}{\varnothing}
\newcommand{\abs}[1]{\ensuremath{\lvert#1\rvert}}

\newcommand{\complexitylangformat}[1]{\fontfamily{cmr}\textup{\textsc{#1}}}

\newcommand{\T}{^{\mathsf T}}
\newcommand{\vtr}[1]{\begin{pmatrix} #1\vtrarg}
	\newcommand{\vtrarg}{\@ifnextchar\bgroup{\vtrnextarg}{\end{pmatrix}}}
\newcommand{\vtrnextarg}[1]{ \\#1\@ifnextchar\bgroup{\vtrnextarg}{\end{pmatrix}}}

\newcommand{\vtrtarg}{\@ifnextchar\bgroup{\vtrtnextarg}{\right)}}
\newcommand{\vtrtnextarg}[1]{ \;#1\@ifnextchar\bgroup{\vtrtnextarg}{\right)^\T}}

\allowdisplaybreaks[1]

\usepackage[algosection,ruled,resetcount,noend]{algorithm2e}

\makeatletter
\patchcmd{\@algocf@start}
{-1.5em}
{0pt}
{}{}
\makeatother

\renewcommand{\SetProgSty}[1]{\renewcommand{\ProgSty}[1]{\textnormal{\csname#1\endcsname{##1}}\unskip}}
\SetProgSty{}
\SetFuncSty{textsc}
\SetFuncArgSty{}
\SetArgSty{}
\SetCommentSty{emph}
\DontPrintSemicolon

\SetKwComment{tcp}{\textup{\texttt{\#}} }{}
\SetKwComment{tcc}{''' }{ '''}
\SetKwProg{Def}{def}{:}{}

\newcommand{\algass}{\ensuremath{\leftarrow}}

\renewcommand{\O}[1]{\ensuremath{\mathcal{O}\left(#1\right)}}

\SetKw{Continue}{continue}
\SetKw{Break}{break}

\usepackage{tikz}
\usetikzlibrary{arrows,petri,topaths}
\usetikzlibrary{decorations,patterns}
\usetikzlibrary{backgrounds}
\usetikzlibrary{fit}
\usetikzlibrary{graphs}
\usetikzlibrary{calc,positioning}
\usetikzlibrary{fadings}
\usetikzlibrary{intersections}
\usetikzlibrary{scopes}
\usetikzlibrary{shapes}

\tikzset{node/.style={circle,draw=black,fill=black,scale=0.5}}
\tikzset{edge/.style={very thick}}
\tikzset{arc/.style={edge,->}}
\tikzset{labelnode/.style={anchor=base}}
\tikzset{highlightpath/.style={draw,line width=8pt,red!20,line cap=round,line join=round}}
\tikzset{>=latex'}

\tikzset{
	dot diameter/.store in=\dot@diameter,
	dot diameter=2pt,
	dot spacing/.store in=\dot@spacing,
	dot spacing=4pt,
	dots/.style={
		line width=\dot@diameter,
		line cap=round,
		dash pattern=on 0pt off \dot@spacing
	}
}
\tikzset{dotted/.style={very thick,dot diameter=2pt,dot spacing=4pt,dots}}

\usepackage{pgfplots}
\pgfplotsset{compat=1.18}

\usepackage[capitalise,noabbrev]{cleveref}

\crefname{enumpropi}{Property}{Properties}
\crefname{enumpropii}{Property}{Properties}
\crefname{enumcondi}{Condition}{Conditions}
\crefname{enumcondii}{Condition}{Conditions}
\crefname{enumstepi}{Step}{Steps}
\crefname{enumstepii}{Step}{Steps}
\crefname{enumparti}{Part}{Parts}
\crefname{enumpartii}{Part}{Parts}
\crefname{enumopi}{Operation}{Operations}
\crefname{enumopii}{Operation}{Operations}
\crefname{enummodi}{Modification}{Modifications}
\crefname{enummodii}{Modification}{Modifications}

\crefname{conjecture}{Conjecture}{Conjectures}
\crefname{convention}{Convention}{Conventions}
\crefname{notation}{Notation}{Notations}
\crefname{observation}{Observation}{Observations}
\crefname{exercise}{Exercise}{Exercises}
\crefname{note}{Note}{Notes}
\crefname{claim}{Claim}{Claims}
\crefname{problem}{Problem}{Problems}
\crefname{assumption}{Assumption}{Assumptions}
\crefname{namedtheorem}{Theorem}{Theorems}
\crefname{namedtheoremcite}{Theorem}{Theorems}
\crefname{namedlemma}{Lemma}{Lemmas}
\crefname{namedlemmacite}{Lemma}{Lemmas}
\crefname{namedconjecture}{Conjecture}{Conjectures}

\crefalias{definitionenum}{definition}
\crefalias{conjectureenum}{conjecture}
\crefalias{lemmaenum}{lemma}
\crefalias{exampleenum}{example}
\crefalias{theoremenum}{theorem}
\crefalias{corollaryenum}{corollary}
\crefalias{conventionenum}{convention}
\crefalias{notationenum}{notation}
\crefalias{remarkenum}{remark}
\crefalias{observationenum}{observation}
\crefalias{exerciseenum}{exercise}
\crefalias{lemmacite}{lemma}
\crefalias{theoremcite}{theorem}
\crefalias{remarkcite}{remark}
\crefalias{notecited}{note}
\crefalias{noteenum}{note}
\crefalias{claimenum}{note}

\usepackage[bibencoding=utf8,style=alphabetic,backend=biber]{biblatex}
\DefineBibliographyExtras{british}{} 

\renewbibmacro*{in:}{} 
\renewbibmacro*{urldate}{} 
\DeclareFieldFormat{titlecase}{#1} 
\DeclareFieldFormat{edition}{#1 edition}  

\usepackage[short]{optidef}

\usepackage{todonotes}

\newcommand{\functDef}[2]{%
	\ifthenelse{\isempty{#2}}{%
		\ensuremath{#1}%
	}{%
		\ensuremath{#1\left(#2\right)}%
	}%
}

\newcommand{\feasibleset}{\ensuremath{X}}
\newcommand{\outcomeset}{\ensuremath{Y}}
\newcommand{\objective}[2][]{\ifthenelse{\isempty{#1}}{\functDef{f}{#2}}{\functDef{f_{#1}}{#2}}}
\newcommand{\problemp}[1][2]{\ensuremath{\Pi_{#1}}}
\newcommand{\problempws}[2][2]{\ensuremath{\Pi_{#1}^{\text{WS}}(#2)}}
\newcommand{\vect}[1]{\ensuremath{\begin{pmatrix}#1\end{pmatrix}}}
\newcommand{\conv}[1]{\text{conv}\left(#1\right)}

\newcommand{\calB}{\ensuremath{\mathcal{B}}}
\newcommand{\calC}{\ensuremath{\mathcal{C}}}
\newcommand{\calE}{\ensuremath{\mathcal{E}}}
\newcommand{\calF}{\ensuremath{\mathcal{F}}}
\newcommand{\calI}{\ensuremath{\mathcal{I}}}
\newcommand{\calN}{\ensuremath{\mathcal{N}}}
\newcommand{\calP}{\ensuremath{\mathcal{P}}}
\newcommand{\calR}{\ensuremath{\mathcal{R}}}
\newcommand{\calS}{\ensuremath{\mathcal{S}}}
\DeclareMathOperator{\rank}{rank}
\newcommand{\adjrel}{\ensuremath{\calR}}
\newcommand{\adjgraph}{\ensuremath{D_\adjrel}}
\newcommand{\esa}{extremely-supportive algorithm}
\newcommand{\tindep}{\ensuremath{T_\text{indep}(M)}}
\newcommand{\matrestr}[2][M]{\ensuremath{#1 | #2}}
\newcommand{\matcontr}[2][M]{\ensuremath{#1 / #2}}
\newcommand{\bmwb}{\complexitylangformat{bmwb}}
\title{An adjacency-based algorithm for computing all extreme-supported non-dominated points of a bi-objective combinatorial optimisation problem}
\author[1]{%
	Oliver Bachtler
	\orcidlink{0000-0001-7942-0750}%
	\thanks{Corresponding author, o.bachtler@math.rptu.de}\textsuperscript{,}%
}
\author[1]{%
	Felix Fritz
	\orcidlink{0009-0000-1392-6556}%
}
\author[1]{%
	Stefan Ruzika
	\orcidlink{0000-0002-3230-0900}%
}
\affil[1]{RPTU University Kaiserslautern-Landau, Kaiserslautern, Germany}
\date{}
\begin{document}

\maketitle

\begin{abstract}
	Generally, multi-objective optimisation problems are solved exactly or approximated by solving a series of scalarisations, for example by dichotomic search.
	In this paper, we take a different approach and attempt to compute the set of all extreme-supported non-dominated points of a bi-objective combinatorial optimisation problem by using a neighbourhood-based approach.
	Whether or not this works depends on the definition of adjacency and we provide sufficient conditions that guarantee its success.
	The resulting generic algorithm is an alternative to dichotomic search in our setting.
	
	We then apply our generic algorithm to a specific example:
	the bi-objective minimum weight basis problem, in which we are given a matroid and want to find bases of minimum weight.
	We use the natural definition of adjacency, in which two bases are adjacent if they differ in exactly one element.
	Since this satisfies our sufficient condition on the adjacency relation, our generic algorithm works in this case and we analyse its running time, showing that it is polynomial.
	By tailoring this algorithm specifically to matroids, we obtain one that is faster but no longer transitions between adjacent solutions, instead swapping directly from one extreme-supported point to the next in a combinatorial fashion.
	
\end{abstract}

\section{Introduction}
\label{sec:intro}

In multi-objective optimisation, problems have multiple conflicting objectives and, thus, they rarely have a single optimal solution.
Instead they have several, potentially very many, solutions whose objective values are not naturally comparable.
Since the corresponding single-objective problems are usually far better studied, scalarisation is the typical approach for solving such problems.
Scalarisation methods take a multi-objective problem and systematically turn it into a single-objective one, which potentially has some parameters or additional constraints.
For surveys of scalarisation in multi-objective optimisation, we refer to \cite{Jah85,MM02,EW05}.

The simplest and most common scalarisation method, and the one we use here, is the weighted-sum scalarisation method, in which each objective is given a weight and these are added up while the feasible set remains unchanged~\cite[Chapter~3]{Ehr05}.
This method typically cannot compute all non-dominated points~$\outcomeset_N$, but those on the boundary of the convex hull of the outcome set: 
the supported non-dominated points~$\outcomeset_{SN}$.
Of these, the extreme points~$\outcomeset_{ESN}$ are of special interest since
\begin{itemize}
	\item they provide an optimal point for every weighted-sum scalarisation problem~\cite{Bac26,PGE10}.
	\item They form a $\Set{(2,1),(1,2)}$-approximation of the set of all non-dominated points~\cite{Bac26,BRTV21}, that is, for every non-dominated point, there is an extreme-supported non-dominated point that is at least as good in one objective and off by at most a factor of 2 in the other.
	\item Based on an empirical study~\cite{Say24}, they provide a good representation of the entire non-dominated set, based on their coverage error and hypervolume ratio.
\end{itemize}

Plenty of algorithms exist that compute the set $\outcomeset_{ESN}$, which we call \emph{\esa s}, most of which concern multi-objective (integer) linear programs.
The most common and general one for two objectives is dichotomic search~\cite{AN79}, which proceeds by solving weighted sum scalarisation problems for certain weights that depend on the previously found images.
An extension of this to more than two objectives can be found in \cite{PGE10}, which solves the higher-dimensional cases by recursively reducing them to the bi-objective case.
Note that our algorithm (and any other one for the bi-objective case) can be inserted into this recursive procedure.
For further examples, see~\cite{Ben98,BS02,OK10,ELS11,HDPR20,BPST24}.

All of these algorithms operate in the objective space to obtain their solutions.
On the other hand, methods that operate in the decision space are more rare and often less efficient due to the fact that the decision space usually has a higher dimension than the objective space.
One prominent example of a decision space method is the parametric Simplex algorithm~\cite{SG54,Ehr05}, which transitions between adjacent basic solutions to find $\outcomeset_{ESN}$.

We want to use definitions of adjacency in this paper to develop an adjacency-based \esa.
While this has not been done generically, concepts of adjacency have been studied before.
Note that we focus on discrete problems, so we are not concerned with topological connectivity, for which results exist as well, see \cite{Nac78} as an example, which is applicable to multi-objective linear programs.

In combinatorial problems, connectivity is usually defined with respect to an adjacency graph.
The vertices of this graph are the feasible solutions and two vertices are connected by an edge if the two solutions are adjacent, for whatever the definition of adjacency is.
These graphs are still typically connected, but for algorithmic purposes, it would be desirable for the efficient solutions (so the subgraph induced by them) to be connected, which is rarely the case.
To be precise, it is not the case for the minimum spanning tree problem and the shortest path problem by \cite{EK97}, nor for the bi-objective integer minimum cost flow problem~\cite{PGE06}.
The efficient solutions of the multi-objective knapsack, travelling salesman, and linear assignment problems are not connected either, except for certain special cases~\cite{dSCF04,Gor10,GKR11}.

After all these negative results, there is some good news:
the supported efficient solutions are often connected.
This is true for basic solutions in linear programming (where they are equal to the basic efficient solutions)~\cite{Ise77} but also for minimum bases matroids~\cite{Ehr96}.

There is no reason to assume that the extreme-supported efficient solutions are connected however and, indeed, they are not for the minimum spanning tree problem, as we shall see.
But, since the supported efficient solutions typically are, we shall develop an algorithm for bi-objective combinatorial problems that generically computes $\outcomeset_{ESN}$ while only visiting points that are supported.

The basic idea is similar to the combinatorial argument used in \cite{Ehr96} to show that the supported efficient bases of a matroid are efficient for the natural definition of adjacency, where two bases are adjacent if they differ in exactly one element.
We start with a lexicographically optimal solution, which is optimal for the second objective and, among those, for the first.
Then, we look at the solutions \enquote{to the left}, so ones with better first objective and transition to one with the flattest slope.
In this way, we traverse the Pareto front and do not skip the extreme-supported non-dominated points.

This is computationally possible, since our feasible set is finite, but very expensive.
To improve this, we do not check all solutions that are \enquote{to the left} of our current one, but just those that are also adjacent.
This limits the solutions we need to check to those in the neighbourhood, but may also cause the result to be incorrect.
We provide sufficient conditions on the adjacency relation that ensure that we do not make mistakes by not considering some solutions.
We also strengthen the criterion slightly to obtain a sufficient criterion for the connectedness of the supported efficient solutions, which can be used to obtain the corresponding result in~\cite{Ehr96}, for example.

Consequently, we obtain a framework which takes an adjacency relation and, if this relation satisfies our sufficient criterion, yields an \esa{} for the problem.
In particular, if we use adjacent feasible solutions to define the adjacency relation, our framework yields the parametric Simplex algorithm.

Afterwards, we look at the bi-objective problem of computing a minimum weight basis of a matroid, which is a large problem class and contains, for example, the minimum spanning tree problem.
We use the adjacency definition mentioned earlier and verify that it satisfies our sufficient condition, yielding a working algorithm.
We prove that it runs in polynomial time, despite the fact that there may be exponentially many supported efficient solutions.

Finally, we use ideas from parametric optimisation to obtain an algorithm tailored to this problem class.
We give up on transitioning between adjacent solutions and instead show how we can directly traverse the extreme-supported non-dominated points in the order \enquote{from bottom right to top left}.
We compare this theoretically to dichotomic search, for which our algorithms are substitutes.

\paragraph{Contributions and outline.}
After covering preliminary definitions and results in \cref{sec:prelims}, we develop a generic \esa{} for bi-objective combinatorial problems in \cref{sec:algorithm}, which can be seen as an alternative to dichotomic search for this problem class.
We show how to incorporate a notion of adjacency in order to make it computationally feasible and obtain the parametric Simplex algorithm as a special case. 
Additionally, we provide sufficient conditions on the adjacency relation that guarantee that the algorithm works and that the supported efficient solutions are connected.
In \cref{sec:matroid-alg}, we look at the bi-objective minimum weight basis problem for matroids and apply our framework to it.
We prove that the resulting \esa{} runs in polynomial time before we improve it in \cref{sec:tailored-alg}, where we also compare it to dichotomic search.
\section{Preliminaries}
\label{sec:prelims}

In this section, we introduce the basic concepts we need for the remainder of this paper.
We start by defining notions of multi-objective optimisation, then we introduce matroids and some of their properties, and finally, we look at what adjacency means in the context of multi-objective optimisation in general and for matroids in particular.

\subsection{Multi-objective optimisation}
A \emph{combinatorial bi-objective optimisation problem} $\problemp$ is of the following form:
\begin{mini*}{}{\objective{x}=(\objective[1]{x},\objective[2]{x})\T}{}{}\tag{\problemp}
    \addConstraint{x}{\in \feasibleset} 
\end{mini*}
where $\objective[1]{},\objective[2]{}\colon \feasibleset \rightarrow \RR$ are the objective functions and $\feasibleset$ is finite.
Let $\outcomeset=\objective{\feasibleset}$ be the outcome set.

In this paper, relations like $\leq$ and $<$ are always interpreted component-wise and we write $y \lneq y'$ if $y\leq y'$ but $y\neq y'$.
We use the Pareto-concept of optimality, in which a point $y$ is called \emph{dominated} by $y'$ if and only if $y' \lneq y$. 
A point $y\in \outcomeset$ which is not dominated by any other point is called a \emph{non-dominated} point and the set of non-dominated points is denoted by $\outcomeset_N$.  
We define $\feasibleset_E \coloneqq \Set{ x \in \feasibleset : \objective{x}\in\outcomeset_N }$ as the set of efficient solutions.

A typical task in a multi-objective optimisation problem is to find all non-dominated points in the outcome set $\outcomeset$ and to every non-dominated point $y$ at least one efficient solution $x \in \feasibleset_E$ with $\objective{x} = y$.
This is typically done by the $\varepsilon$-constraint approach or by Tchebycheff scalarisation, which transfer the multi-objective optimisation problem to a single objective one.
Another typical scalarisation, that cannot yield all non-dominated points in general, is \emph{weighted-sum problems}, which minimises $w\T\objective{x}$.
By normalising $w$, we can always assume that $w_1 + w_2 = 1$ and it suffices to specify one component, say $w_1 = \lambda$ (and $w_2=1-\lambda$).
Thus, for $\lambda \in [0,1]$, we define $\problempws{\lambda}$ by
\begin{mini*}{}{\lambda\objective[1]{x} + (1-\lambda)\objective[2]{x}}{}{}\tag{\problempws{\lambda}}
    \addConstraint{x}{\in \feasibleset}{.}
\end{mini*}
We write $\feasibleset_\lambda$ for all optimal solutions of $\problempws{\lambda}$.

A solution $x\in \feasibleset_\lambda$, for $\lambda\in(0,1)$, is a \emph{supported efficient solution} and its image $y$ is a \emph{supported non-dominated point}.
If $y$ is an extreme point of $\conv{Y}+\RR_{\geq 0}$, then $y$ is an \emph{extreme-supported non-dominated point} and $x$ an \emph{extreme-supported efficient solution}.
We write $\feasibleset_{SE}$, $\feasibleset_{ESE}$, $\outcomeset_{SN}$, and $\outcomeset_{ESN}$ for the sets of such solutions and points.
The solutions $x\in \feasibleset_{SE}$ are efficient and $y\in\outcomeset_{SN}$ are non-dominated~\cite[Theorem~3.6]{Ehr05}.

We note that several different definitions of supported efficient solutions and supported non-dominated points are used in the literature.
For an overview we refer to~\cite{Chl25,KS25}, which show that they all coincide in our setting of bi-objective combinatorial problems.

Let us look at a small but illustrative example.
\begin{example}
	\label{ex:moo-problem-ST}
	Let $G$ be the graph on five vertices shown in \cref{fig:moo-problem-ST-ex}.
	We want to find a minimum spanning tree with respect to the costs depicted therein.
	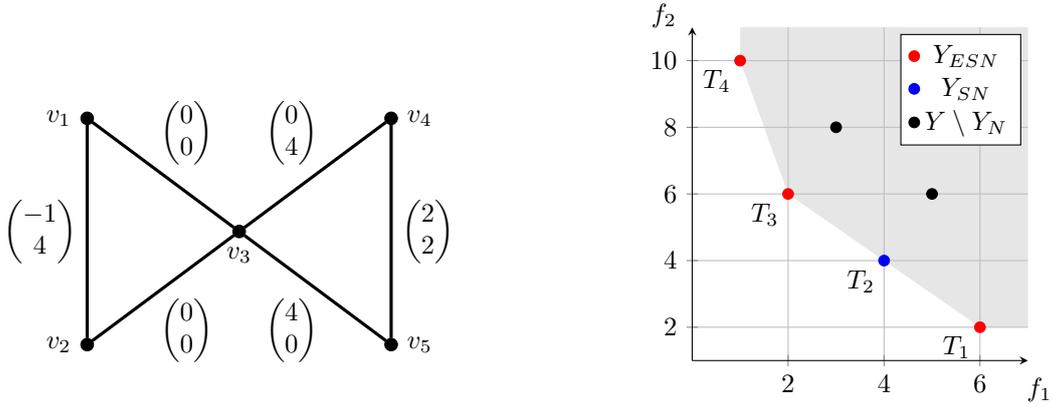
\begin{figure}[tb]
		\centering
		\subcaptionbox{An example graph for the bi-objective minimum spanning tree problem.
		\label{fig:moo-problem-ST-ex}}[.475\textwidth]
		{
			\begin{tikzpicture}
				\node[node,label=left:$v_1$] (1) at (0, 1.5) {};
				\node[node,label=left:$v_2$] (2) at (0, -1.5) {};
				\node[node,label=below:$v_3$] (3) at (2, 0) {};
				\node[node,label=right:$v_4$]  (4) at (4, 1.5) {};
				\node[node,label=right:$v_5$]  (5) at (4, -1.5) {};
				
				\node (dummy) at (0,-2.25) {};
				
				\draw[edge] (1) to node[midway, left] {$\vect{-1\\4}$} (2);
				\draw[edge] (1) to node[midway, above right,xshift=-5pt] {$\vect{0\\0}$} (3);
				\draw[edge] (2) to node[midway, below right,xshift=-5pt] {$\vect{0\\0}$} (3);
				\draw[edge] (3) to node[midway, above left,xshift=5pt] {$\vect{0\\4}$} (4);
				\draw[edge] (3) to node[midway, below left,xshift=5pt] {$\vect{4\\0}$} (5);
				\draw[edge] (4) to node[midway, right] {$\vect{2\\2}$} (5);
			\end{tikzpicture}
		}
		\hfill
		\subcaptionbox{The costs of all spanning trees.
		The grey area is $\conv{\outcomeset}+\RR_{\geq 0}$.
		\label{fig:moo-problem-ST-costs}}[.475\textwidth]
		{
			\begin{tikzpicture}
				\begin{axis}[
					axis lines=middle,    
					xlabel={$\objective[1]{}$},
					x label style={at={(axis description cs:1.1,-0.15)}},
					ylabel={$\objective[2]{}$},
					y label style={at={(axis description cs:-0.15,1.1)}},
					xmin=0, xmax = 7,     
					ymin=1, ymax = 11,   
					grid=major,     
					width=6cm, height=6cm
					]
					\addplot [only marks, red] coordinates {
						(1,10)
						(2,6)
						(6,2)
					};
					\addlegendentry{$\outcomeset_{ESN}$}
					\coordinate[label = below left:{$T_1$}] (t1) at (6,2);
					\coordinate[label = below left:{$T_3$}] (t3) at (2,6);
					\coordinate[label = below left:{$T_4$}] (t4) at (1,10);
					
					\addplot [only marks, blue] coordinates {
						(4,4)
					};
					\coordinate[label = below left:{$T_2$}] (t2) at (4,4);
					\addlegendentry{$\outcomeset_{SN}$}
					
					\addplot [only marks] coordinates {
						(5,6)
						(3,8)
					};
					\addlegendentry{$Y\setminus \outcomeset_N$}
					
					\begin{scope}[on background layer]
						\addplot[fill=black!10,draw=none] coordinates {(2,6) (1,10) (1,11) (7,11) (7,2) (6,2)} --cycle;
					\end{scope}
				\end{axis}
			\end{tikzpicture}
		}
		\caption{An example bi-objective minimum spanning tree problem.}
	\end{figure}
	
	This graph has nine spanning trees and their corresponding costs are shown in \cref{fig:moo-problem-ST-costs}.
	The efficient spanning trees correspond to the red and blue dots.
	The trees $T_1$, $T_2$, and $T_3$ all contain the edges $v_1v_3$ and $v_2v_3$ as well as the edges $v_3v_5$ and $v_4v_5$, $v_3v_4$ and $v_3v_5$, or $v_3v_4$ and $v_4v_5$.
	The tree $T_4$ uses the edges $v_1v_2$, $v_1v_3$, $v_3v_4$, and $v_4v_5$, though the same point is attained by the tree $T_4'\coloneqq T_4 - v_1v_3 + v_2v_3$.
	The remaining trees $T_5$, $T_5'$ and $T_6$, $T_6'$ have objective values $\vect{5,6}\T$ and $\vect{3,8}\T$, which are dominated by the images of $T_2$ or $T_3$.
\end{example}

We also need the definition of the weight set decomposition (with respect to the weighted sum scalarisation), which we shall briefly recall.
We denote the weight set $[0,1]$ by $\Lambda$ and write $\Lambda(x)$ for the weight set component of $x$, that is, $\Lambda(x) = \Set{\lambda\in\Lambda : x \in \feasibleset_{\lambda}}$.
We regularly use the following result concerning the weight set, for which we refer to \cite[Theorem~3.14]{Bac26} (which is based on \cite{PGE10}).
\begin{lemma}
	\label{prelims-ese-multiple-weights}
	The weight set component $\Lambda(x)$ of $x\in \feasibleset$ consists of more than a single weight if and only if $x\in\feasibleset_{ESE}$.
\end{lemma}

\subsection{Matroids and the Greedy algorithm}
Matroids are a generalisation of spanning trees and are the structure in which the Greedy algorithm is optimal.
For the basics of matroids, we refer to \cite{Oxl11}, but we briefly state the definitions and properties we need here.
\begin{definition}
	\label{def:matroid}
	A \emph{matroid} $M$ is an ordered pair $(E, \calI)$, where $E$ is a finite set and $\calI \subseteq 2^E$ satisfies the following three properties: 
	\begin{enumprop}
	    \item $\emptyset\in\calI$.
	    \item If $I \in \calI$ and $I' \subseteq I$, then $I' \in \calI$.
	    \item If $I_1,\, I_2 \in \calI$ and $\abs{I_1} < \abs{I_2}$, then there exists an element $e \in I_2 \setminus I_1$ such that $I_1 \cup \Set{e} \in \calI$.\label{prop:prelim-matroid-exchange}
	\end{enumprop}
\end{definition}
The elements of $\calI$ are called \emph{independent sets} of $M$, while all other subsets of $E$ are called \emph{dependent}. 
A \emph{basis} of $M$ is a maximal independent set of $M$ and a \emph{circuit} is a minimal dependent set.
We denote the set of all bases of $M$ by~$\calB(M)$ and the set of all circuits by~$\calC(M)$.
All bases have the same cardinality which is called the \emph{rank} $\rank(M)$ of $M$.
\begin{example}
    Let $E$ be the edge set of a connected graph $G$ and let $\calF$ be the set of all acyclic subsets $F \subseteq E$. 
    Then $M = (E, \calF)$ is called a \emph{graphic matroid}.
    The bases of $M$ are the spanning trees, the circuits are cycles, and the rank of $M$ is $n-1$ where $n$ is the order of the graph $G$.
\end{example}

For some single-objective cost function $c\colon E \rightarrow \RR$, we can find a basis with minimal cost by applying the Greedy algorithm (\cref{alg:Greedy}).
It performs $m$ independence tests, the complexity of which may differ from matroid to matroid.
Thus, we denote the time to perform an independence check in $M$ by $\tindep$ and obtain a running time of $\O{m\cdot(\log m + \tindep)}$.
\begin{algorithm}[tb]
	Sort the elements of $E=\Set{e_1, e_2, \ldots, e_m}$ such that $c(e_1) \leq c(e_2) \leq \ldots \leq c(e_m)$\;
	Initialize $I \algass \emptyset$\;
	\For{$i=1,\ldots,m$}
	{
		\If{$I \cup \Set{e_i} \in \calI$}
		{
			$I \algass I \cup \Set{e_i}$\;
		}
	}
	\Return{$I$}
	\caption{The Greedy algorithm for a matroid $M$.}
	\label{alg:Greedy}
\end{algorithm}

That this algorithm is correct is a key property of matroids.
In fact, the Greedy algorithm is correct for every choice of costs $c$ if and only if the $M$ is a matroid~\cite[Theorem~1.8.5]{Oxl11}.

The following strong basis exchange property will be helpful for us later.
\begin{lemmacite}[{{\cite[Theorem~2]{Bru69}}}]
	\label{prelims-matroid-bases-swap-elements}
	Let $B$ and $B'$ be bases of $M$.
	For every $e\in B$ there exists an $e'\in B'$ such that $(B-e)+e'$ and $(B'-e')+e$ are bases of $M$.
\end{lemmacite}

Our main application in this paper will be the \emph{bi-objective minimum weight basis problem} (\bmwb), which is defined as follows:
\begin{mini*}{}{\objective{B}=(\objective[1]{B},\objective[2]{B})\T}{}{}\tag{\bmwb}
    \addConstraint{B}{\in \mathcal{B}(M)} 
\end{mini*}
where, for $i\in\Set{1,2}$, $\objective[i]{B} = \sum_{e \in B}c_i(e)$ with $c\colon E\to\RR^2$.
Note that we can solve the weighted-sum problem $\problempws{\lambda}$ in this case by applying the Greedy algorithm with weights $c_\lambda$, where $c_\lambda(e) \coloneqq \lambda c_1(e) + (1-\lambda)c_2(e)$ for every $\lambda\in\Lambda$.

\subsection{Definitions and results concerning adjacency}
Finally, we want to define what it means for certain solutions of our combinatorial optimisation problem $\problemp$ to be connected.
To do so, we need a notion of adjacency between different solutions, which is just a binary relation.
\begin{definition}
	Let $\adjrel$ be a binary relation on \feasibleset.
	A feasible solution $x\in\feasibleset$ is \emph{adjacent} to a feasible solution $x'\in\feasibleset$ if $(x,x')\in\adjrel$.
	
	The \emph{adjacency graph} of \adjrel{} is the digraph $\adjgraph$ with vertex set $X$ and arc set~\adjrel.
	We treat $\adjgraph$ as an undirected graph if $\adjrel$ is symmetric.
\end{definition}

Now, a set of solutions $\feasibleset'\subseteq\feasibleset$ is (weakly or strongly) connected if $\adjgraph[\feasibleset']$ is.
To shorten notation, we write $D_E$, $D_{SE}$, and $D_{ESE}$ for $\adjgraph[\feasibleset']$ if $\feasibleset'$ is $\feasibleset_E$, $\feasibleset_{SE}$, and $\feasibleset_{ESE}$.

As an example, for matroids, we use the natural definition of adjacency where two bases are adjacent if they differ in exactly one element.
\begin{definition}
	\label{def:adj-matroids}
	Let $M$ be a matroid and $\adjrel_M$ be the relation on $\calB(M)$ containing all pairs of bases $B,\, B'$ such that $\abs{B\setminus B'} = 1$.
\end{definition}
This is a symmetric relation, so we can regard the adjacency graph $D\coloneqq D_{\calR_M}$ as undirected.
For graphic matroids, this just means that adjacent spanning trees differ in exactly one edge.
If we recall \cref{ex:moo-problem-ST}, the adjacency graph for that problem is shown in \cref{fig:moo-problem-ST-adjacency}.
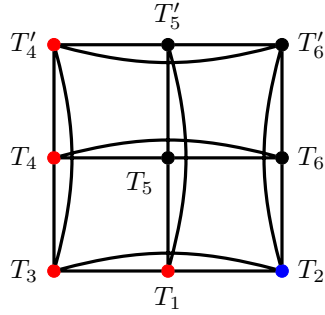
\begin{figure}[tb]
	\centering
	\begin{tikzpicture}
		\node[node,label=below: $T_1$,red]  (1)  at (1.5, 0  ) {};
		\node[node,label=right: $T_2$,blue] (2)  at (3,   0  ) {};
		\node[node,label=left:  $T_3$,red]  (3)  at (0,   0  ) {};
		\node[node,label=left:  $T_4$,red]  (4)  at (0,   1.5) {};
		\node[node,label=left:  $T_4'$,red] (4p) at (0,   3  ) {};
		\node[node,label=below left: $T_5$] (5)  at (1.5, 1.5) {};
		\node[node,label=above: $T_5'$]     (5p) at (1.5, 3  ) {};
		\node[node,label=right: $T_6$]      (6)  at (3,   1.5) {};
		\node[node,label=right: $T_6'$]     (6p) at (3,   3  ) {};
		
		\draw[edge] (1) to (2)
					(1) to (3)
					(1) to (5)
					(1) to[bend right=15] (5p)
					(2) to[bend right=15] (3)
					(2) to (6)
					(2) to[bend left=15] (6p)
					(3) to (4)
					(3) to[bend right=15] (4p)
					(4) to (4p)
					(4) to (5)
					(4) to[bend left=15] (6)
					(5) to (5p)
					(5) to (6)
					(6) to (6p)
					(4p) to (5p)
					(4p) to[bend right=15] (6p)
					(5p) to (6p);
	\end{tikzpicture}
	\caption{The adjacency graph of the MST problem from \cref{ex:moo-problem-ST}.}
	\label{fig:moo-problem-ST-adjacency}
\end{figure}

Several results concerning the connectivity of certain subgraphs of $D$ are known, though not many of them are positive.
By the exchange property (\cref{prop:prelim-matroid-exchange} in \cref{def:matroid}), we can transform every basis $B$ into every other basis $B'$ by iteratively exchanging single elements, which yields that $D$ is connected.
But these exchanges could lead to intermediate bases that are not efficient for $\problemp$, indeed, $D_E$ need not be~\cite{EK97}.

\Textcite{Ehr96} showed that $D_{SE}$ is connected and we shall generalise this proof and obtain it again as a special case of our result.
In addition, we shall develop an adjacency-based algorithm that computes $\outcomeset_{ESN}$ that works in the case of matroids.
However, the following example shows that $D_{ESE}$ is not connected.
\begin{example}
	\label{prelim-ex-ese-not-connected}
	Regard the graph from \cref{fig:moo-problem-ST-ex}, where we replace the costs of the left triangle by the same costs as on the right.
	By doing so, we obtain the objective values shown in \cref{fig:moo-problem-ST-costs-mod}.
	Note that the missing trees have the following values: $\objective{T_4'}=\objective{T_5}=\objective{T_2}$, $\objective{T_6'}=\objective{T_1}$, and $\objective{T_6}=\objective{T_3}$.
	
	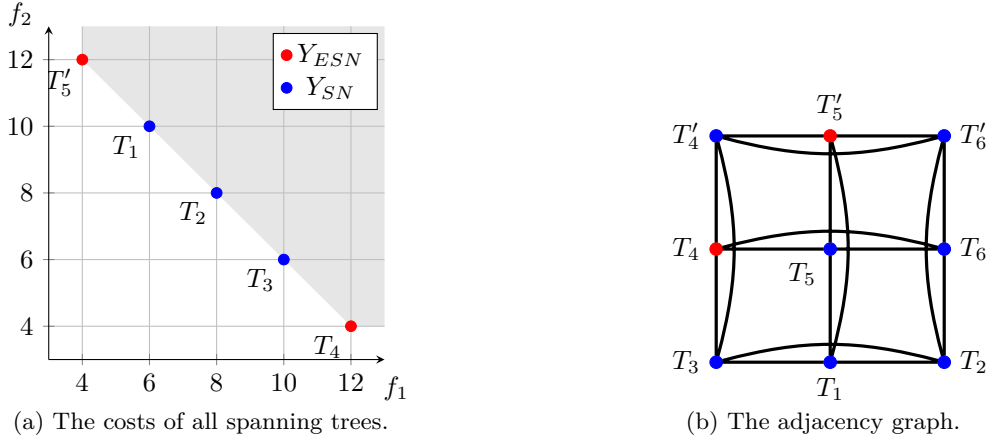
\begin{figure}[tb]
		\centering
		\subcaptionbox{The costs of all spanning trees.
		\label{fig:moo-problem-ST-costs-mod}}[.475\textwidth]
		{
			\begin{tikzpicture}
				\begin{axis}[
					axis lines=middle,    
					xlabel={$\objective[1]{}$},
					x label style={at={(axis description cs:1.1,-0.15)}},
					ylabel={$\objective[2]{}$},
					y label style={at={(axis description cs:-0.15,1.1)}},
					xmin=3, xmax = 13,     
					ymin=3, ymax = 13,   
					grid=major,     
					width=6cm, height=6cm
					]
					\addplot [only marks, red] coordinates {
						(12,4)
						(4,12)
					};
					\addlegendentry{$\outcomeset_{ESN}$}
					\coordinate[label = below left:{$T_4$}] (t4) at (12,4);
					\coordinate[label = below left:{$T_5'$}] (t5p) at (4,12);
					
					\addplot [only marks, blue] coordinates {
						(8,8)
						(10,6)
						(6,10)
					};
					\coordinate[label = below left:{$T_2$}] (t2) at (8,8);
					\coordinate[label = below left:{$T_3$}] (t3) at (10,6);
					\coordinate[label = below left:{$T_1$}] (t1) at (6,10);
					\addlegendentry{$\outcomeset_{SN}$}

					\begin{scope}[on background layer]
						\addplot[fill=black!10,draw=none] coordinates {(12,4) (4,12) (4,13) (13,13) (13,4)} --cycle;
					\end{scope}
				\end{axis}
			\end{tikzpicture}
		}
		\hfill
		\subcaptionbox{The adjacency graph.
		\label{fig:moo-problem-ST-adjacency-mod}}[.475\textwidth]
		{
			\begin{tikzpicture}
				\node[node,label=below: $T_1$,blue]  	 (1)  at (1.5, 0  ) {};
				\node[node,label=right: $T_2$,blue] 	 (2)  at (3,   0  ) {};
				\node[node,label=left:  $T_3$,blue]  	 (3)  at (0,   0  ) {};
				\node[node,label=left:  $T_4$,red]  	 (4)  at (0,   1.5) {};
				\node[node,label=left:  $T_4'$,blue] 	 (4p) at (0,   3  ) {};
				\node[node,label=below left: $T_5$,blue] (5)  at (1.5, 1.5) {};
				\node[node,label=above: $T_5'$,red]      (5p) at (1.5, 3  ) {};
				\node[node,label=right: $T_6$,blue]      (6)  at (3,   1.5) {};
				\node[node,label=right: $T_6'$,blue]     (6p) at (3,   3  ) {};
				
				\draw[edge] (1) to (2)
				(1) to (3)
				(1) to (5)
				(1) to[bend right=15] (5p)
				(2) to[bend right=15] (3)
				(2) to (6)
				(2) to[bend left=15] (6p)
				(3) to (4)
				(3) to[bend right=15] (4p)
				(4) to (4p)
				(4) to (5)
				(4) to[bend left=15] (6)
				(5) to (5p)
				(5) to (6)
				(6) to (6p)
				(4p) to (5p)
				(4p) to[bend right=15] (6p)
				(5p) to (6p);
			\end{tikzpicture}
		}
		\caption{The costs and adjacency graph of the spanning trees in \cref{prelim-ex-ese-not-connected}.}
	\end{figure}
	Here we see that all these points lie on one line and the extreme-supported efficient solutions are exactly the trees $T_4$ and $T_5'$, which are not adjacent, see~\cref{fig:moo-problem-ST-adjacency}. 
\end{example}

As the example shows, an algorithm that computes $\outcomeset_{ESN}$ while moving between adjacent solutions in $\feasibleset_{ESE}$ cannot exist.
Thus, we shall need to allow our algorithm to visit intermediate solutions that are not in $\feasibleset_{ESE}$, but we shall make sure that it does not leave $\feasibleset_{SE}$.
Therefore, we call an algorithm that computes a set $S\subseteq\feasibleset_{ESE}$ such that $\objective{S} = \outcomeset_{ESN}$ an \emph{\esa}.

Do note, however, that $\feasibleset_{SE}$ can contain exponentially many elements:
if we look at the graphic matroid for the complete graph whose edges all have the same cost, then all $n^{n-2}$ spanning trees~\cite{AZ18} are optimal bases with respect to every weight vector.
This result can be extended to an example where $\outcomeset_N$ is exponentially large as well, see \cite[Theorem~4.2]{HR94}.
\section{A generic \esa}
\label{sec:algorithm}

The goal of this section is to develop a generic adjacency-based \esa{} for a combinatorial bi-objective problem \problemp.
Our algorithm starts with a lexicographically optimal solution $x^0$ with respect to the ordering $(\objective[2]{},\objective[1]{})$, which is extreme-supported efficient~\cite[Lemma~2.2]{Bac26}.
\begin{figure}
	\centering
	\begin{tikzpicture}
		\begin{axis}[
			axis lines=middle,    
			xlabel={$\objective[1]{}$},
			x label style={at={(axis description cs:1.1,-0.15)}},
			ylabel={$\objective[2]{}$},
			y label style={at={(axis description cs:-0.15,1.1)}},
			xmin=0, xmax = 7,     
			ymin=1, ymax = 11,   
			grid=major,     
			width=6cm, height=6cm
			]
			\addplot [only marks] coordinates {
				(1,10)
				(2,6)
				(6,2)
			};
			\coordinate[label = above right:{$T_1$}] (t1) at (6,2);
			\coordinate[label = above right:{$T_3$}] (t3) at (2,6);
			\coordinate[label = right:{\;$T_4$}] (t4) at (1,10);
			
			\addplot [only marks] coordinates {
				(4,4)
			};
			\coordinate[label = above right:{$T_2$}] (t2) at (4,4);
			
			\addplot [only marks] coordinates {
				(5,6)
				(3,8)
			};
			
			\begin{scope}[on background layer]
				\addplot[fill=black!10,draw=none] coordinates {(2,6) (1,10) (1,11) (7,11) (7,2) (6,2)} --cycle;
			\end{scope}
			
			\draw[arc] (axis cs:6,2) to[bend left] (axis cs:2,6);
			\draw[arc, dashed] (axis cs:6,2) to[bend right] (axis cs:4,4);
			\draw[arc, dashed] (axis cs:4,4) to[bend right] (axis cs:2,6);
			\draw[arc] (axis cs:2,6) to[bend left] (axis cs:1,10);
		\end{axis}
	\end{tikzpicture}
	\caption{Possible steps for an \esa.}
	\label{fig:algorithm-illustration}
\end{figure}
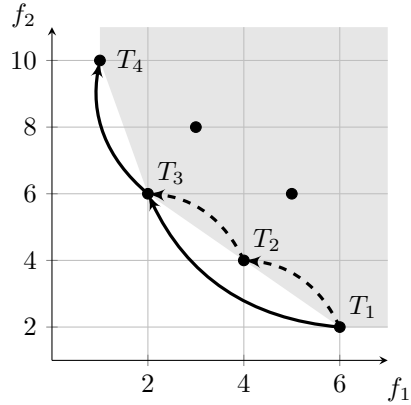

If we recall \cref{ex:moo-problem-ST}, the solution $x^0$ is the spanning tree $T_1$.
For the sake of convenience, \cref{fig:algorithm-illustration} shows the outcome set of the example again.
Here, the weight set $\Lambda(T_1) = [0,\tfrac{1}{2}]$.
The algorithm will now increase $\lambda$ from $0$ until $\tfrac{1}{2}$, at which points there are multiple optimal solutions to $\problempws{\tfrac{1}{2}}$, namely $T_1$, $T_2$, and $T_3$.
At this point, as indicated in the figure, the algorithm would want to transition to $T_3$ (or to $T_2$ and then $T_3$ since we allow intermediate points in $\feasibleset_{SE}$ to be visited).

The way this is realised is by noting that $T_2$ and $T_3$ are exactly those spanning trees, for which the line between their objective value and $\objective{T_1}$ is least steep.
Thus, the algorithm looks at the points that have lower $\objective[1]{}$-value and transition to the one that leads to the least steep slope and then continues onward from that point.
We formally prove that this approach works in the first subsection, which gives us a \emph{global algorithm}.

The downside of this algorithm is that it needs to look at lots of points in every step.
We remedy this by not looking at all points that have lower $\objective[1]{}$-value, but only at those that are adjacent to the current point.
This will not work in general, since we might miss all the solutions with the least steep slope in this manner.
Thus, the algorithm only works for specific notions of adjacency.
In the second subsection, we provide a sufficient condition for the adjacency relation that guarantees the algorithm's correctness.
There, we will also see that adjacent basic solutions of linear programs satisfy this condition and that our algorithm in this case is exactly the parametric Simplex algorithm.

\subsection{A global algorithm}
As we noted above, we want to transition from a solution $x$ to a solution $x'$ of lower $\objective[1]{}$-value for which the slope $\alpha$ between $\objective{x}$ and $\objective{x'}$ is least steep.
Such a slope $\alpha$ is directly related to the value $\lambda$ for which both $x$ and $x'$ are optimal:
this is the case when $w\coloneqq(\lambda,1-\lambda)\T$ is a normal vector of the line of slope $\alpha$.
A simple computation shows that this is equivalent to $\lambda=\tfrac{\alpha}{\alpha-1}$.
Thus, $w\T\objective{x}$ is lower than $w\T\objective{x'}$ before this $\lambda$ and higher after, making $\lambda$ an upper bound on the weight set component of $x$.
Consequently, to find the upper bound of $\Lambda(x)$ (at which point we get another optimal solution), we need to find the maximal slope, which leads to the smallest bound.

Analogously, solutions $x'$ with higher $\objective[1]{}$-value provide us with lower bounds on the weight set component in the same way, where the highest, so the one corresponding to the minimal slope, is the lower bound of $\Lambda(x)$.
That this works is the result of \cref{algorithm-weight-set}, for which the following definition gives us convenient access to slopes and their corresponding parameters in the weight set.
\begin{definition}
	Let $x\in \feasibleset$.
	For $x'\in \feasibleset$, we define
	\begin{displaymath}
		\vect{\Delta_1(x',x)\\\Delta_2(x',x)} \coloneqq  \objective{x'} - \objective{x}.
	\end{displaymath}
	If $\Delta_1(x',x)\neq 0$, we define $s(x',x) \coloneqq \tfrac{\Delta_2(x',x)}{\Delta_1(x',x)}$.
	In addition, we set $\feasibleset^<(x) \coloneqq \Set{x'\in \feasibleset : \Delta_1(x',x) < 0}$ and $\feasibleset^>(x) \coloneqq \Set{x'\in \feasibleset : \Delta_1(x',x) > 0}$.
	
	Finally, for $\lambda\in[0,1)$ we define $\alpha(\lambda)$ as $\tfrac{-\lambda}{1-\lambda} \in (-\infty,0]$.
	Conversely, for $\alpha\in (-\infty,0]$, we define $\lambda(\alpha)$ as $\tfrac{\alpha}{\alpha-1} \in [0,1)$.
	We extend the notation by setting $\alpha(1) \coloneqq -\infty$ and $\lambda(-\infty) \coloneqq 1$.
\end{definition}

Before we do the computation, we note the following two properties of $\alpha$ and $\lambda$.
First, they are inverse functions and, second, they invert inequalities, for example, if $\lambda\leq \lambda'$, then $\alpha(\lambda)\geq \alpha(\lambda')$.
We shall make use of these properties regularly.
\begin{lemma}
	\label{algorithm-weight-set}
	Let $x\in \feasibleset_{SE}$.
	Then $\Lambda(x)=\left[\lambda(\alpha^x),\lambda(\alpha_x)\right]$, where
	\begin{align*}
		\alpha_x &\coloneqq \max \Set{s(x',x) : x' \in \feasibleset^<(x)}\cup\Set{-\infty},\\
		\alpha^x  &\coloneqq \min\, \Set{s(x',x) : x' \in \feasibleset^>(x)}\cup\Set{0}.
	\end{align*}
	Additionally, if $s(x',x) = \alpha_x$ for $x'\in \feasibleset^<(x)$, then $x'\in \feasibleset_{\lambda(\alpha_x)}$ and if $s(x',x) = \alpha^x$ for $x'\in \feasibleset^>(x)$, then $x'\in \feasibleset_{\lambda(\alpha^x)}$.
	In both cases, $x'$ is supported efficient.
\end{lemma}
\begin{proof}
	Let $x,\, x'\in\feasibleset$ with $\Delta_1(x',x) < 0$ and $\lambda\in(0,1)$, $w\coloneqq(\lambda,1-\lambda)\T$.
	Then
	\begin{align}\nonumber
		w\T\objective{x} \leq w\T\objective{x'}
		&\iff (1-\lambda)\Delta_2(x',x) \geq -\lambda\Delta_1(x',x) \\\nonumber
		&\iff (1-\lambda)s(x',x) \leq -\lambda \\\nonumber
		&\iff s(x',x) \leq \frac{-\lambda}{1-\lambda} = \alpha(\lambda)\\
		&\iff \lambda(s(x',x)) \geq \lambda.\label{eq:x-better-x'-slopes}
	\end{align}
	The same computation works with \enquote{$\geq$} and \enquote{$=$}.
	In particular, if $\Delta_1(x',x) > 0$, then
	\begin{displaymath}
		w\T\objective{x} \leq w\T\objective{x'} \iff \lambda(s(x',x)) \leq \alpha(\lambda).
	\end{displaymath}
	Since $x\in\feasibleset_{\lambda}$ if and only if $w\T\objective{x} \leq w\T\objective{x'}$ for all $x'\in\feasibleset$, we can use the computation above to reformulate the right hand side to $\lambda \leq\lambda(s(x',x))$ for $x'\in\feasibleset^<(x)$, $\lambda \geq\lambda(s(x',x))$ for $x'\in\feasibleset^>(x)$, and $\objective[2]{x} \leq \objective[2]{x'}$ if $\Delta_1(x',x)=0$.
	
	If $x$ is supported efficient, then the third condition is satisfied and thus, by the definition of $\alpha_x$ and $\alpha^x$, we get that $x\in\feasibleset_{\lambda}$ if and only if $\lambda\in\left[\lambda(\alpha^x),\lambda(\alpha_x)\right]$.
	As a result, $\Lambda(x)$ has the desired form.
	
	For the \enquote{additionally} part, we regard the case that $s(x',x) = \alpha_x = \alpha(\lambda(\alpha_x))$.
	Using the equality version of \eqref{eq:x-better-x'-slopes}, we see that this implies that $x$ and $x'$ have the same objective value for $\lambda(\alpha_x)$.
	Since $x\in\feasibleset_{\lambda(\alpha_x)}$, so is $x'$.
		
	To see that $x'$ is supported efficient, we note that $\Delta_1(x',x) < 0$ and, since $x$ was efficient, $\Delta_2(x',x) > 0$.
	Hence, $\alpha_x\in(-\infty,0)$ and $\lambda(\alpha_x)\in(0,1)$.
	The case for $\alpha^x$ is analogous.
\end{proof}

Before we formulate our algorithm, let us note two simple consequences of this lemma.
The first tells us that a transition from $x$ to a solution $x'$ to the left of maximal slope is next in the weight set decomposition in the sense that the weight set component of $x'$ starts where the component of $x$ ended.
The second tells us that the slope between $\objective{x}$ and $\objective{x'}$ must be equal to $\alpha(\lambda)$ if both $x$ and $x'$ are optimal for $\problempws{\lambda}$.
\begin{corollary}
	\label{algorithm-lb-up-alpha-relation}
	Let $x\in\feasibleset_{SE}$, $x'\in\feasibleset^<(x)$, and $\alpha_x = s(x,x')$.
	Then, $\alpha^{x'} = \alpha_x$.
\end{corollary}
\begin{proof}
	Since $x'\in\feasibleset^<(x)$, we get $x\in\feasibleset^>(x')$ and $\alpha^{x'} \leq s(x,x') = \alpha_x$.
	But since both $x$ and $x'$ are in $X_{\lambda(\alpha_x)}$ by \cref{algorithm-weight-set}, we have that $\lambda(\alpha_x)\in\Lambda(x')$.
	Hence, $\lambda(\alpha_x) \geq \lambda(\alpha^{x'})$, which yields $\alpha_x \leq \alpha^{x'}$, giving us the desired equality.
\end{proof}

\begin{corollary}
	\label{algorithm-slopes-coincide-for-same-lambda}
	Let $\lambda\in\Lambda$, $x\in\feasibleset_\lambda$, and $x'\in\feasibleset$.
	Then $s(x',x) \leq \alpha(\lambda)$ if $x'\in\feasibleset^<(x)$ and $s(x',x) \geq \alpha(\lambda)$ is $x'\in\feasibleset^>(x)$.
	In particular, if $x'\in\feasibleset_\lambda$ and $\objective{x}\neq\objective{x'}$, then $\alpha(\lambda) = s(x',x)$.
\end{corollary}
\begin{proof}
	If $x'\in\feasibleset^<(x)$, we get $s(x',x) \leq \alpha_x \leq \alpha(\lambda)$ and if $x'\in\feasibleset^>(x)$, then $s(x',x) \geq \alpha^x \geq \alpha(\lambda)$ by \cref{algorithm-weight-set}.
	For the \enquote{in particular} part, note that we can exchange the roles of $x$ and $x'$ in this case.
	We may assume that $x'\in\feasibleset^<(x)$, which yields $\alpha(\lambda) \leq s(x',x) = s(x,x') \leq \alpha(\lambda)$.
\end{proof}

We can now formulate the very simple \cref{alg:ESN-simple}:
it starts with the lexicographically optimal solution $x^0$ as described and, in iteration $k$, determines the solution $x^{k}$ in $\feasibleset^<(x^{k-1})$ that maximises the slope $s(x^{k},x^{k-1})$ and continues looking for points from $x^k$.
\begin{algorithm}[tb]
	Let $x^0$ be a lexicographically optimal solution with respect to $(\objective[2]{},\objective[1]{})$\;
	$S \algass \Set{x^0}$, $k\algass 1$\;
	\While{$\feasibleset^<(x^{k-1}) \neq \emptyset$}
	{
		Let $x^{k} \in \arg\max\Set{s(x',x^{k-1}) : x'\in \feasibleset^<(x^{k-1})}$\;
		$S \algass S\cup\Set{x^{k}}$, $k\algass k+1$\;
	}
	\Return{$S$}
	\caption{A generic \esa{}.
		\label{alg:ESN-simple}}
\end{algorithm}

Based on the previous lemma and its corollaries, we can make a few simple observations.
\begin{observation}
	\label{algorithm-properties-of-simple-alg}
	For points $x^0,\ldots,x^l$ computed by \cref{alg:ESN-simple} the following properties hold:
	\begin{enumerate}
		\item $x^0\in\feasibleset_{SE}$, because it is a lexicographic solution.
		\item For $k\in\Set{1,\ldots,l}$, $s(x^{k},x^{k-1}) = \alpha_{x^{k-1}} = \alpha^{x^k} \eqqcolon \alpha^k$, by \cref{algorithm-lb-up-alpha-relation}.
		\item For $k\in\Set{1,\ldots,l}$, both $x^{k-1}$ and $x^k$ are supported efficient and in $\feasibleset_{\lambda^k}$ for $\lambda^k\coloneqq \lambda(\alpha^k)$, inductively by \cref{algorithm-weight-set}.
		\item For $k\in\Set{2,\ldots,l}$, $\alpha^{k-1} = \alpha^{x^{k-1}} \geq \alpha_{x^{k-1}} = \alpha^{k}$ and thus $\lambda^{k-1} \leq \lambda^{k}$.
	\end{enumerate}
\end{observation}

We now have everything we need to prove correctness.
\begin{theorem}
	\label{algorithm-simple-alg-correctness}
	\Cref{alg:ESN-simple} is correct, that is, it returns a set $S\subseteq \feasibleset_{SE}$ such that $\outcomeset_{ESN}\subseteq \objective{S}$.
\end{theorem}
\begin{proof}
	By \cref{algorithm-properties-of-simple-alg}, we know that all the points $x^0,\ldots,x^l$ computed are in $\feasibleset_{SE}$.
	Thus, we only need to show that we compute at least one preimage for every point in $\outcomeset_{ESN}$.
	
	To this end, let $y\in \outcomeset_{ESN}$, then $y=\objective{x}$ for an $x\in \feasibleset_{ESE}$.
	We know that $\objective[2]{x^0} \leq \objective[2]{x}$ and, since $x$ is efficient, $\objective[1]{x} \leq \objective[1]{x^0}$.
	Furthermore, $\objective[1]{x^l} \leq \objective[1]{x}$ since the algorithm terminated at $x^l$.
	
	If $\objective[1]{x} = \objective[1]{x^k}$ for some $k$, then $\objective[2]{x} = \objective[2]{x^k}$ since both $x$ and $x^k$ are efficient solutions and $y=\objective{x^k}\in \objective{S}$.
	Otherwise, $\objective[1]{x} \in\left(\objective[1]{x^k},\objective[1]{x^{k-1}}\right)$ for some $1\leq k\leq l$.
	In this case, we can apply \cref{algorithm-properties-of-simple-alg}: since $x\in\feasibleset^<(x^{k-1})$, $\alpha^k = \alpha_{x^{k-1}} \geq s(x,x^{k-1}) = s(x^{k-1},x)$.
	Analogously, $x\in\feasibleset^>(x^k)$ implies $\alpha^k = \alpha^{x^k} \leq s(x,x^{k})$ as well.
	
	Combining this with the fact that $x^k\in\feasibleset^<(x)$ and $x^{k-1}\in\feasibleset^>(x)$, we conclude that $\alpha_x \geq s(x^k,x) \geq \alpha^k \geq s(x^{k-1},x) \geq \alpha^x$, which yields $\Lambda(x) \subseteq \Set{\lambda^k}$, a contradiction to \cref{prelims-ese-multiple-weights}.
\end{proof}

Also note that we could run the algorithm starting at an arbitrary solution $x^0\in\feasibleset_{SE}$ and we would still obtain all extreme-supported efficient solutions with a higher $\objective[2]{}$-value.
This is true since the only properties of $x^0$ that we used is that $x^0$ is supported efficient and that $\objective[2]{x^0} \leq y_2$ for $y\in\outcomeset_{ESN}$.
In summary, we get the following corollary.
\begin{corollary}
	\label{algorithm-simple-alg-start-midway}
	\Cref{alg:ESN-simple}, when initialised with an arbitrary solution $x^0\in\feasibleset_{SE}$, computes a set $S\subseteq \feasibleset_{SE}$ with $\outcomeset_{ESN}\cap\Set{y\in\RR^2 : y_2\geq \objective[2]{x^0}} \subseteq \objective{S}$.
\end{corollary}

\Cref{alg:ESN-simple} can easily be modified to return a set $S\subseteq\feasibleset_{ESE}$ by simply not adding the non-extreme-supported efficient solutions we find.
These can be determined on the fly since we have implicitly computed the weight set of each of the points we added:
by \cref{algorithm-weight-set,algorithm-properties-of-simple-alg},
\begin{displaymath}
	\Lambda(x^k) = \left[\lambda(\alpha^{x^k}),\lambda(\alpha_{x^k})\right] = \left[\lambda^{k},\lambda^{k+1}\right],
\end{displaymath}
for $k\in\Set{1,\ldots,l-1}$.
For the first and last solution, we get $\Lambda(x^0) = \left[0,\lambda^1\right]$ and $\Lambda(x^l) = \left[\lambda^l,1\right]$.
This makes it easy to check whether the weight set contains at most one element and the point can be discarded by \cref{prelims-ese-multiple-weights}.
Additionally, the algorithm also computes a weight set decomposition.

\subsection{Obtaining an adjacency-based version}
Now that we have our base outline, we want to use adjacency.
The basic idea is to not determine the next solution by finding the maximum slope among all other solutions, but just amongst those in the neighbourhood (so amongst the successors in the adjacency graph).
To this end, we let $D$ be the adjacency graph and denote the set $\feasibleset^<(x) \cap N_D^{\text{out}}(x)$ by $N^<(x)$.
The new algorithm, which is illustrated in \cref{alg:ESN-adj-based-generic}, now simply replaces $\feasibleset^<(x)$ by $N^<(x)$.
\begin{algorithm}[tb]
	Let $x^0$ be a lexicographically optimal solution with respect to $(\objective[2]{},\objective[1]{})$\;
	$S \algass \Set{x^0}$, $k\algass 1$\;
	\While{$N^<(x^{k-1}) \neq \emptyset$}
	{
		Let $x^{k} \in \arg\max\Set{s(x',x^{k-1}) : x'\in N^<(x^{k-1})}$\;
		$S \algass S\cup\Set{x^{k}}$, $k\algass k+1$\;
	}
	\Return{$S$}
	\caption{A generic adjacency-based \esa{}.
	\label{alg:ESN-adj-based-generic}}
\end{algorithm}

To make this work, we need to ensure that our notion of adjacency does not make us miss essential points when looking in this restricted set.
This is done by the following theorem.
\begin{theorem}
	\label{algorithm-generic-adjacency-alg-correctness}
	\Cref{alg:ESN-adj-based-generic} is correct if, for all $\lambda\in(0,1)$ and all $x\in\feasibleset_\lambda$, the set $\feasibleset^<(x)\cap\feasibleset_\lambda$ is either empty or contains a successor of $x$ in $D$.
\end{theorem}
\begin{proof}
	We show inductively that the algorithm is a possible execution of \cref{alg:ESN-simple}, letting us transfer the correctness from \cref{algorithm-simple-alg-correctness}.
	More precisely, we show that all solutions chosen up to iteration $k$ could also have been candidates in \cref{alg:ESN-simple} and that both terminate at the same time.
	
	The first part is clearly true before the first iteration.
	Hence, we may assume the claim holds until iteration $k-1$ and want to show that the chosen element $x^k$ is also a possible choice for \cref{alg:ESN-simple}, that is, $s(x^k,x^{k-1}) = \alpha^k$.
	Since $N^<(x)\neq\emptyset$, $X^<(x)$ is as well and \Cref{alg:ESN-simple} chooses a solution $\tilde{x}^k$ with $s(\tilde{x}^k,x^{k-1}) = \alpha_x$.
	Using \cref{algorithm-properties-of-simple-alg}, we recall that $x^{k-1},\,\tilde{x}^k\in\feasibleset_{\lambda^k}$.
	Hence, $\tilde{x}^k\in\feasibleset^<(x)\cap\feasibleset_{\lambda^k}$, meaning this set contains a successor $\tilde{x}$ of $x$.
	By \cref{algorithm-slopes-coincide-for-same-lambda} and the choice of $x^k$,
	\begin{displaymath}
		s(x^k,x^{k-1}) \geq s(\tilde{x},x^{k-1}) = \alpha^k = \alpha_{x^{k-1}} \geq s(x^k,x^{k-1}),
	\end{displaymath}
	 proving the desired equality.
\end{proof}
Again, this proof works for the modified algorithm which starts with an arbitrary supported efficient point $x^0$.

We now take a look at an example, which shows that the parametric simplex algorithm is a special case of our result, at least if the bi-objective linear program it solves is non-degenerate.
\begin{example}[The parametric Simplex algorithm]
	The Simplex method~\cite{NW88} is a method for solving linear programs.
	Its extension to the bi-objective setting is the parametric simplex algorithm~\cite{Ehr05}, which solves the problem
	\begin{mini*}{}{\vect{c_1\T x\\c_2\T x}}{}{}
		\addConstraint{Ax}{= b}
		\addConstraint{x}{\geq 0}
	\end{mini*}
	where $A\in\RR^{m\times n}$ with $\rank A = m$, $b\in\RR^m$, and $c_1,\, c_2\in\RR^n$.
	For our presentation, we adhere closely to \cite[Section~6.2]{Ehr05}, with the exception that the algorithm there goes from $\lambda=1$ to $\lambda=0$, so we swap the objectives and thus $\lambda$ with $(1-\lambda)$.
	
	A basis $\calB$ of $A$ is a selection of $m$ linearly independent columns of $A$, that is, the submatrix $A_\calB$ containing these columns is invertible.
	We denote $A_\calB^{-1}A$ by $\tilde{A}$ and $A_\calB^{-1}b$ by~$\tilde{b}$.
	If $\tilde{b}\geq 0$, the basis is feasible and the vector of reduced costs is given by $\overline{c}\T = c\T - c_\calB\T\tilde{A}$.
	We need to assume that the problem is non-degenerate, that is, $\tilde{b} > 0$ for all bases $\calB$.
	
	The Simplex algorithm then starts with a feasible basis $\calB$ and then iteratively chooses a variable $x_i$ to enter the basis that has negative reduced cost $\overline{c}_i$.
	It then determines how far $x_i$ can be increased before a variable $x_j$ with $j\in\calB$ reaches $0$ and one such variable then leaves the basis.
	If $x_i$ can be increased arbitrarily, the problem is infeasible and if no variable with negative reduced cost remains, the solution is optimal.
	
	The parametric Simplex now uses the Simplex algorithm to find optimal bases for weighted sum scalarisations for increasing $\lambda$, see \cref{alg:parametric-simplex}.
	It starts by computing a lexicographically optimal basis with respect to $(\objective[2]{},\objective[1]{})$.
	We obtain reduced costs~$\overline{c}^1$ and $\overline{c}^2$, of which $\overline{c}^2\geq 0$.
	Note that the reduced costs for $\problempws{\lambda}$ are exactly $\lambda\overline{c}^1 + (1-\lambda)\overline{c}^2$.
	Thus, if $\overline{c}^1\geq 0$ holds as well, $x$ is also optimal for all larger $\lambda$ and the algorithm terminates.
	Otherwise it finds the first $\lambda$ after which a new entry of negative reduced cost is created, which becomes the entering variable and the leaving variable is computed as in the Simplex algorithm.
	\begin{algorithm}[tb]
		Let $\calB^0$ be a lexicographically optimal basis with respect to $(\objective[2]{},\objective[1]{})$\;
		$S = \Set{\calB^0}$, $k\algass 1$\;
		$\calI \algass \Set{i\notin \calB^0 : \overline{c}^1_i < 0}$\;
		\While{$\calI\neq \emptyset$}
		{
			$\lambda\algass\min\Set{\frac{-\overline{c}^1_i}{\overline{c}^2_i-\overline{c}^1_i} : i\in\calI}$ and $s\in\arg\min\Set{\frac{-\overline{c}^1_i}{\overline{c}^2_i-\overline{c}^1_i} : i\in\calI}$\;
			$r\in \arg\min\Set{j\in\calB^{k-1}: \frac{\tilde{b}_j}{\tilde{A}_{js}} : \tilde{A}_{js}>0}$\;
			$\calB^k \algass \calB^{k-1}\cup\Set{s}\setminus\Set{r}$, $\calI \algass \Set{i\notin \calB^{k} : \overline{c}^1_i < 0}$\;
			$S\algass S\cup\Set{\calB^k}$, $k\algass k+1$\;
		}
		\Return{$S$}
		\caption{The parametric Simplex algorithm.
		\label{alg:parametric-simplex}}
	\end{algorithm}
	
	To see that this is exactly \cref{alg:ESN-adj-based-generic}, we have already formulated it in terms of feasible bases (instead of basic feasible solutions), but every basis~$\calB$ comes with its basic solution $x_\calB$.
	We say two feasible bases $\calB$ and $\calB'$ are adjacent if they differ in exactly one element.
	Thus, \cref{alg:parametric-simplex} transitions between adjacent bases.
	
	To see that $\calB^k\in N^<(\calB^{k-1})$, we need to take a closer look at this set.
	The set $N^<(\calB^{k-1})$ contains those bases $\calB$ that differ from $\calB^{k-1}$ by one element and that satisfy $(c^1)\T x < (c^1)\T x^{k-1}$ where $x,\, x^{k-1}$ are the basic solutions corresponding to $\calB$ and $\calB^{k-1}$.
	Like in the Simplex algorithm, 
	\begin{equation}
		\label{eq:reduced-costs-simplex}
		(c^1)\T x = (c^1)\T x^{k-1} + (\overline{c}^1_\calN)\T x_\calN
	\end{equation}
	where $\calN \coloneqq \Set{1,\ldots,n}\setminus\calB$ and $\overline{c}$ are the reduced costs for $x^{k-1}$.
	
	If we assume that $\calB = \calB^{k-1}\cup\Set{i}\setminus\Set{j}$, then \eqref{eq:reduced-costs-simplex} becomes $(c^1)\T x = (c^1)\T x^{k-1} + \overline{c}^1_i x_i$, since the entering variable $i$ is the only variable in $\calN$ where $x$ has a positive value.
	Consequently, $x\in N^<(x^{k-1})$ if and only if $\overline{c}^1_i < 0$, by our non-degeneracy assumption.
	Thus, the parametric Simplex chooses an element $\calB^k$ in $N^<(x^{k-1})$ in each step and it is a special case of \cref{alg:ESN-adj-based-generic}.
	
	To obtain a proof of correctness from our framework, we need to verify that the condition of \cref{algorithm-generic-adjacency-alg-correctness} is satisfied.
	Hence, let $\lambda\in(0,1)$ and $\calB$ be an optimal basis for $\problempws{\lambda}$.
	We may assume there exists a basis $\calB'\in \feasibleset^<(\calB)$ that is also optimal for this weighted sum scalarisation and now need to show that such a basis exists which is also a neighbour.
	
	Let $x$ and $x'$ be the corresponding basic solutions.
	Since $(c^1)\T x' < (c^1)\T x$, we can use \eqref{eq:reduced-costs-simplex} to obtain
	\begin{displaymath}
		(c^1)\T x > (c^1)\T x' = (c^1)\T x + (\overline{c}^1_{\calN})\T x'_{\calN}.
	\end{displaymath}
	Thus, for some $i\in\calN$, $\overline{c}_i x'_i < 0$, that is, $\overline{c}_i < 0$ and $x'_i > 0$.
	The latter gives us $i\in \calB'\setminus\calB$.
	Let $\calB^*$ be a feasible basis obtained by letting $i$ enter the basis $\calB$ and $x^*$ be the corresponding basic solution.
	Then, by \eqref{eq:reduced-costs-simplex}, $\overline{c}_i < 0$, and the non-degeneracy assumption, $(c^1)\T x^* = (c^1)\T x + \overline{c}_i^1x^*_i < (c^1)\T x$.
	Thus, $\calB^*\in N^<(\calB)$.
	
	To complete the proof, we need to show that $\calB^*$ is optimal for $\problempws{\lambda}$.
	Let $c^\lambda\coloneqq \lambda c^1 + (1-\lambda) c^2$ and $\overline{c}^\lambda\coloneqq \lambda\overline{c}^1 + (1-\lambda)\overline{c}^2$.
	Since both $x$ and $x'$ are optimal for $c^\lambda$, we can conclude that
	\begin{displaymath}
		(c^\lambda)\T x 
		= (c^\lambda)\T x' 
		\stackrel{\eqref{eq:reduced-costs-simplex}}{=} (c^\lambda)\T x + (\overline{c}^\lambda_{\calN})\T x'_\calN
		= (c^\lambda)\T x + (\overline{c}^\lambda_{\calB'\setminus \calB})\T x'_{\calB'\setminus\calB}.
	\end{displaymath}
	By the optimality of $x$ for $c^\lambda$ and the non-degeneracy assumption, we get that $\overline{c}^\lambda_{\calN}\geq 0$, so $\overline{c}^\lambda_{\calB'\setminus\calB} = 0$.
	In particular, $\overline{c}_i^\lambda=0$ and $c^\lambda x^* \stackrel{\eqref{eq:reduced-costs-simplex}}{=} (c^\lambda)\T x + \overline{c}^\lambda_{i} x^*_i = (c^\lambda)\T x$.
	Thus, $\calB^*$ is optimal.
\end{example}

We end this section on a connectivity result for $D_{SE}$, which we can obtain by slightly strengthening the condition in \cref{algorithm-generic-adjacency-alg-correctness}.
\begin{theorem}
	\label{algorithm-sufficient-condition-for-connectivity-alg}
	If the condition from \cref{algorithm-generic-adjacency-alg-correctness} is met and the lexicographically optimal solutions for the ordering $(\objective[1]{},\objective[2]{})$ are weakly connected, then so is $D_{SE}$.
\end{theorem}
\begin{proof}
	Let $x\in\feasibleset_{SE}$ be an arbitrary point.
	By \cref{algorithm-generic-adjacency-alg-correctness,algorithm-simple-alg-start-midway}, \cref{alg:ESN-adj-based-generic} computes a sequence of adjacent solutions $x^0,\ldots,x^l$ such that $\outcomeset_{ESN}\cap\Set{y\in\RR^2 : y_2\geq \objective[2]{x^0}}\subseteq \objective{\Set{x^0,\ldots,x^l}}$.
	In particular, since for $x^l$ satisfies that $\feasibleset^<(x^l)=\emptyset$, it must be a lexicographically optimal solution for the ordering $(\objective[1]{},\objective[2]{})$.
	Since these solutions are weakly connected by assumption and every supported efficient solution is connected to one of them, $D_{SE}$ is connected.
\end{proof}

An alternative (and more natural) criterion requires the optimal solutions to each weighted sum scalarisation to be connected.
\begin{theorem}
	\label{algorithm-sufficient-condition-for-connectivity}
	Then the graph $D_{SE}$ is (weakly or strongly) connected if $D[\feasibleset_\lambda]$ is (weakly or strongly) connected for all $\lambda\in(0,1)$.
\end{theorem}
\begin{proof}
	By \cite[Corollary~3.6]{Bac26} or \cite[Proposition~6]{PGE10}, we know that $\Lambda$ is given by the union of intervals $[0,\lambda^1],[\lambda^1,\lambda^2],\ldots,[\lambda^k,1]$ with $0<\lambda^1<\ldots<\lambda^k<1$ that originate from extreme-supported efficient solutions, say $\Lambda(x^i) = [\lambda^i,\lambda^{i+1}]$ for $i\in\Set{1,\ldots,k-1}$ and $x^i\in\feasibleset_{ESE}$ (by \cref{prelims-ese-multiple-weights}).
	Note that $x^i\in \feasibleset_{\lambda^i}\cap \feasibleset_{\lambda^{i+1}}$.
	Since $D[\feasibleset_{\lambda^i}]$ is connected for all $i$, so is $\bigcup_{i=1}^k D[\feasibleset_{\lambda^i}]\subseteq D[\bigcup_{i=1}^k \feasibleset_{\lambda^i}]$.
	But $\bigcup_{i=1}^k \feasibleset_{\lambda^i}=\feasibleset_{SE}$ by \cite[Theorem~3.9]{Bac26}, from which the result follows.
\end{proof}
Note that it would have been sufficient to require $D[\feasibleset_\lambda]$ to be connected for $\lambda\in\Set{\lambda^1,\ldots,\lambda^k}$, we did not use connectivity for the remaining $\lambda$-values.
\section{A polynomial time adjacency-based \esa{} for {\normalfont\bmwb}}
\label{sec:matroid-alg}

In this section we want to make use of the generic \Cref{alg:ESN-adj-based-generic} to compute the extreme-supported non-dominated points for the minimum weight basis in a matroid in polynomial time.
Thus, for the remainder of this section, let $M$ be a matroid, where $M=(E,\calI)$, and let $D$ be the adjacency graph for the notion of adjacency defined in \cref{def:adj-matroids}.
We first show that for this notion of adjacency the requirements of \cref{algorithm-generic-adjacency-alg-correctness} are met and then discuss how to implement the algorithm in this special case to determine its running time.

\subsection{Applying Algorithm~\ref{alg:ESN-adj-based-generic}}

Using \cref{prelims-matroid-bases-swap-elements}, we can prove that the requirements of \cref{algorithm-generic-adjacency-alg-correctness} are met, giving us an \esa.
\begin{theorem}
	\label{matroid-alg-correctness}
	Let $B\in\feasibleset_\lambda$ such that $\feasibleset^<(B)\cap\feasibleset_\lambda\neq\emptyset$.
	Then there exists a basis $B'\in N^<(B)\cap\feasibleset_\lambda$.
	In particular, \cref{alg:ESN-adj-based-generic} is correct in this case.
\end{theorem}
\begin{proof}
	Let $\lambda$ and $B$ be as in the claim and let $B'$ be a basis in $\feasibleset^<(B)\cap\feasibleset_\lambda$ that minimises $\abs{B\setminus B'}$.
	We show that $\abs{B\setminus B'} = 1$, which yields $B'\in N^<(B)\cap\feasibleset_\lambda$.
	
	Let $e\in B\setminus B'$.
	By \cref{prelims-matroid-bases-swap-elements} we can find an element $e'\in B'$ such that $B_1\coloneqq(B-e)+e'$ and $B_2\coloneqq(B'-e')+e$ are bases.
	Since $B$ and $B'$ are optimal with respect to the costs $c_\lambda(e)$, we can conclude that $c_\lambda(e) = c_\lambda(e')$:
	otherwise, either $B_1$ or $B_2$ would have a better objective value than $B$ or $B'$.
	Hence, $B_1$ and $B_2$ are in $\feasibleset_\lambda$.
	
	We now note that $B_2$ differs from $B$ is one fewer element than $B'$ does.
	So by minimality of $\abs{B\setminus B'}$, $B_2\notin\feasibleset^<(B)$ and $c_1(e) > c_1(e')$.
	Hence, $c_1(B_1) < c_1(B)$ and $B_1\in\feasibleset^<(B)\cap\feasibleset_\lambda$, making it a possible choice for $B'$, so $\abs{B\setminus B'} = 1$.
\end{proof}

Before we implement the algorithm for matroids, let us briefly remark that we can also prove connectivity of the supported efficient solutions easily in this case, giving us an alternative proof of \cite[Theorem~4]{Ehr96}.
\begin{theorem}
	The graph $D_{SE}$ is connected.
\end{theorem}
\begin{proof}
	By \cref{algorithm-sufficient-condition-for-connectivity-alg,matroid-alg-correctness} it suffices to show that the lexicographically optimal bases are connected.
	But since any such basis can be obtained by the Greedy algorithm (that sorts elements of identical weight such that the basis elements come first), these are connected.
	Simply exchange pairs of elements of identical weight to transform one sorting into the other.
\end{proof}
Note that the argument suffices to use \cref{algorithm-sufficient-condition-for-connectivity} as well, when applied to all orderings $c_\lambda$.

\subsection{Implementing the algorithm}
But what is the running time of the algorithm?
For the analysis, recall that we denote the time to perform an independence check in $M$ by $\tindep$.
Additionally, let $r\coloneqq\rank(M)$ and $m\coloneqq \abs{E}$.

To obtain a solution that is lexicographically optimal, we need to run the Greedy algorithm on the lexicographic sorting, resulting in a running time of $\O{m\cdot(\log m+\tindep)}$.
The remaining time of the algorithm is spent on finding an element $B'\in N^<(B)$ that maximises the slope $s(B',B)$ in each iteration.
Note that a neighbour $B'$ in $N(B)$ is of the form $(B-e)+e'$.
Since $c_1(B) > c_1(B')$, we must have $c_1(e) > c_1(e')$ and we can test all $\O{(m-r)r}$ such pairs.
This yields a running time of $\O{mr\cdot\tindep}$ per iteration.

We remark that this is an approach that works generically, but it might be suboptimal.
An alternative approach is to determine the circuit in $B+e'$, for $e'\notin B$.
Then the potential elements in $B$  we can swap $e'$ with are exactly the elements $e\in C\setminus\Set{e'}$ with $c_1(e) > c_1(e')$.
We can find a circuit in $\O{r\cdot\tindep}$ by removing the elements of $B$ from $B+e'$ if the remaining set is still dependent.
This implementation directly yields our original running time, but we benefit if this can be done faster than $\O{r\cdot\tindep}$.

One example where this is the case is the graphic matroid.
Here, $\tindep\in\O{n}$, but this is also the time required to find the cycle in $T+e'$.
In this way, we could reduce the time needed for enumerating the neighbourhood to $\O{mn}$ from $\O{mn^2}$.
But, for simplicity, we shall restrict the analysis to just the generic running time of $\tindep$.

It remains to bound the number of iterations.
The basic idea is that, when we transition from a basis $B^{k-1}$ to a neighbour $B^k = (B^{k-1}-e)+e'$, then $c_{\lambda^k}(e) = c_{\lambda^k}(e')$, since both bases are in $\feasibleset_{\lambda^k}$ by \cref{algorithm-properties-of-simple-alg}.
So such swaps can only occur for a specific value of $\lambda^k$, which only increases during the course of our algorithm.
A certain swap may occur multiple times, but since we exchange an element $e$ of larger $c_1$-value with an element $e'$ of smaller value, we can still strictly decrease the number of potential swaps that remain in each iteration.
Since these are initially bounded by all possible swaps, the iterations are bounded by $\O{m^2}$.
The following notation and theorem formalises this idea.
\begin{notation}
	Let $\calP\coloneqq\Set{(e,f)\in E^2 : c_1(e) > c_1(f),\, c_2(e) < c_2(f)}$.
	Let $e,\, f\in E$ be two fixed elements with $c_1(e) \geq c_1(f)$.
	Then, $c_\lambda(e)$ and $c_\lambda(f)$ are linear functions in $\lambda$, which intersect in exactly one point $\lambda(e,f)\in(0,1)$ if and only if $(e,f)\in\calP$.
	We write $\calE$ for the set $\Set{\lambda(e,f) : (e,f)\in\calP}$.
\end{notation}

\begin{theorem}
	\Cref{alg:ESN-adj-based-generic} applied to a matroid $M$ terminates after at most $\O{m^2}$ iterations of the while loop.
\end{theorem}
\begin{proof}
	Let $B^0,\ldots,B^l$ be the bases computed by \Cref{alg:ESN-adj-based-generic} applied to $M$.
	Let $(e,f)\in\calP$.
	We call the pair $(e,f)$ \emph{$k$-feasible} if $\lambda^k < \lambda(e,f)$ or if $\lambda^k=\lambda(e,f)$, $e\in B^k$, and $f\notin B^k$.
	Initially, for $k=0$, there are at most $\O{m^2}$ $k$-feasible pairs and we now prove that each iteration strictly reduces this number.
	
	To do so, let $k\in\Set{1,\ldots,l}$ and $B^k = B^{k-1} - e + f$, where $c_1(e) > c_1(f)$ since $B'\in\feasibleset^<(B)$ and, since $B'$ is efficient, $c_2(f) > c_2(e)$.
	By \cref{algorithm-properties-of-simple-alg}, $c_{\lambda^k}(e) = c_{\lambda^k}(f)$, since both bases are in $\feasibleset_{\lambda^k}$, so $\lambda(e,f) = \lambda^k\geq \lambda^{k-1}$ and $\Set{e,f}$ is a $(k-1)$-feasible pair but not $k$-feasible.
	We now show that we can assign a distinct $(k-1)$-feasible pair to every $k$-feasible pair, without using $\Set{e,f}$.
	
	Let $(e',f')$ be a $k$-feasible pair.
	If $\lambda(e',f') > \lambda^k$, then this pair is also $(k-1)$-feasible.
	Thus, we may assume that $\lambda^k = \lambda(e',f')$, $e'\in B^k$ and $f'\notin B^{k}$.
	If $e'\in B^{k-1}$ and $f'\notin B^{k-1}$, then again $(e',f')$ is $(k-1)$-feasible.
	
	This leaves the case that $e'\in B^k\setminus B^{k-1} = \Set{f}$ or $f'\in B^{k-1}\setminus B^k = \Set{e}$.
	In either case, the elements $e,\, f,\, e',\, f'$ all have the same $c_{\lambda^k}$-value.
	If $e' = f$, then $c_1(e) > c_1(f) = c_1(e') > c_1(f')$, so $f'\neq e$ and $f'\notin B^{k-1}$.
	In particular, the pair $(e,f')$ satisfies $\lambda(e,f') = \lambda^k$ and $e\in B^{k-1}\setminus B^k$, $f'\notin B^{k-1}$.
	Thus, the pair $(e,f')$ is $(k-1)$-feasible but not $k$-feasible.
	
	The case $f'=e$ is analogous.
	Here, $c_1(f) < c_1(e) = c_1(f') < c_1(e')$, so $e'\neq f$ and $e'\in B^{k-1}$.
	This time, the pair $(e',f)$ satisfies $\lambda(e',f) = \lambda^k$ and $e'\in B^{k-1}$, $f\in B^k\setminus B^{k-1}$ and the pair is $(k-1)$-feasible but not $k$-feasible.
\end{proof}

With this bound on the number of iterations, we obtain the following result.
\begin{corollary}
	\Cref{alg:ESN-adj-based-generic} can be implemented to run in $\O{m^3r\cdot\tindep}$ time on matroids.
\end{corollary}\section{Tailoring an algorithm to {\normalfont\bmwb}}
\label{sec:tailored-alg}

Again, let $M$ be a matroid with $M=(E,\calI)$.
In this section, we design a faster \esa{} for the bi-objective minimum weight basis problem.
To this end, we shall make use of the Greedy algorithm:
the original lexicographic solution is obtained by running Greedy on the lexicographically sorted edges.
We then saw that edge exchanges happen when edges have the same weight.
By sorting the points $\lambda$ where two edges have the same weight, the adjacency-based algorithm traversed solutions that were in $\feasibleset_\lambda$ until it reached the next extreme point and $\lambda$ increased.
We shall prove here that the extreme points correspond to two specific orderings of the edges:
the first extreme point our algorithm found corresponds to the ordering where edges of equal $c_\lambda$-value are ordered by non-increasing $c_1$-value, while the second extreme point orders by non-decreasing $c_1$-value.
In this way, we can jump directly to the next extreme point.

The following definition provides us with convenient notation to formulate the resulting \cref{alg:ESN-matroid-specific}.
\begin{definition}
	For $\lambda\in\calE$, we define
	\begin{displaymath}
		E_\lambda\coloneqq\bigcup_{\substack{(e,f)\in\calP:\\ \lambda(e,f) = \lambda}}\Set{e,f}.
	\end{displaymath}
\end{definition}
\begin{algorithm}[tb]
	Let $\lambda^1<\ldots<\lambda^s$ be the values in $\calE$\;
	Let $B^0$ be a lexicographically optimal solution with respect to $(\objective[2]{},\objective[1]{})$\;
	$S\algass\Set{B^0}$\;
	\For{$k=1,\ldots,s$}
	{
		$B^k\algass B^{k-1}\setminus E_{\lambda^{k}}$\;
		Let $e^1,\ldots,e^t$ be the elements in $E_{\lambda^{k}}$, sorted lexicographically by $(c_{\lambda^k},c_1)$ in non-descending order\;
		\For{$i=1,\ldots,t$}
		{
			\If{$B^k+e^i$ is independent}{$B^k\algass B^k + e^i$}
		}
		$S\algass S\cup\Set{B^k}$, $k\algass k+1$\; 
	}
	\Return{$S$}
	\caption{An algorithm that computes a representation of $\outcomeset_{ESN}$ for matroids.}
	\label{alg:ESN-matroid-specific}
\end{algorithm}

\subsection{Proving correctness of Algorithm~\ref{alg:ESN-matroid-specific}}
To show that this algorithm works, some more notation is needed.
\begin{definition}
	Let $\calS$ be an ordering of the elements of a matroid $M$.
	We write $B_M(\calS) = B(\calS)$ for the bases obtained by running the Greedy algorithm for this ordering.
	For $\lambda\in[0,1]$, we write $\calS_\lambda$ for a non-descending ordering with respect to $c_\lambda$.
	The two special cases, where the ordering is lexicographic with respect to $(c_\lambda, c_1)$ and $(c_\lambda, -c_1)$, are denoted by $\calS_\lambda^\uparrow$ and $\calS_\lambda^\downarrow$.
\end{definition}
Note that elements of equal cost can always be ordered arbitrarily, but we assume the orderings we regard from now on do this consistently (as does the algorithm).
As a result, orderings $\calS_\lambda$ can only differ in their ordering of the elements in $E_\lambda$.
We start by showing that $S_{\lambda^k}^\uparrow = S_{\lambda^{k+1}}^\downarrow$.
\begin{observation}
	\label{spec-alg-orderings-element-swaps}
	Let $e,\, f\in E$ with $c_1(e)\geq c_1(f)$.
	Additionally, let $\calS_\lambda$ and $\calS_{\lambda'}$ be two orderings with $0\leq \lambda < \lambda' < 1$.
	
	If $(e,f)\in\calP$, then
	\begin{align*}
		c_\lambda(e) &< c_\lambda(f) && \text{ for all } \lambda < \lambda(e,f),\\
		c_\lambda(e) &> c_\lambda(f) && \text{ for all } \lambda > \lambda(e,f),\\
		c_\lambda(e) &= c_\lambda(f) && \text{ at } \lambda = \lambda(e,f).
	\end{align*}
	The ordering of $e$ and $f$ is consistent in $\calS_\lambda$ and $\calS_{\lambda'}$ unless the intersection point $\lambda(e,f)$ lies between $\lambda$ and $\lambda'$, that is, unless $\lambda\leq \lambda(e,f) \leq \lambda'$.
	
	If $(e,f)\notin\calP$, then $c_1(e) = c_1(f)$ or $c_1(e) > c_1(f)$ and $c_2(e) \geq c_2(f)$.
	In the first case, $c_\mu(e)$ and $c_\mu(f)$ intersect at $1$.
	Thus, $c(e) = c(f)$ or $c_\mu(e)$ and $c_\mu(f)$ do not intersect on $[0,1)$, meaning their ordering in $\calS_\lambda$ and $\calS_{\lambda'}$ coincides, by our assumption on the consistent ordering of elements with identical cost.
	This is true in the last missing case as well, where $c_1(e) > c_1(f)$ and $c_2(e) \geq c_2(f)$, since here the cost functions $c_\mu(e)$ and $c_\mu(f)$ do not intersect at all.
	
	In summary, elements $e$ and $f$ are ordered consistently by $\calS_\lambda$ and $\calS_{\lambda'}$ unless $(e,f)\in\calP$ and $\lambda\leq \lambda(e,f) \leq \lambda'$.
\end{observation}

\begin{lemma}
	\label{spec-alg-orderings-for-consecutive-lambdas}
	Let $\lambda^1<\ldots<\lambda^s$ be the values in $\calE$, and $\lambda^0\coloneqq0$.
	Then, for $k\in\Set{1,\ldots,s}$, $\calS_{\lambda^{k-1}}^\uparrow = \calS_{\lambda^{k}}^\downarrow$.
\end{lemma}
\begin{proof}
	Let $e,\, f\in E$ with $c_1(e)\geq c_1(f)$.
	By \cref{spec-alg-orderings-element-swaps}, since $0\leq\lambda^{k-1} < \lambda^{k} < 1$, we know that $\calS_{\lambda^{k-1}}^\uparrow$ and $\calS_{\lambda^{k}}^\downarrow$ order $e$ and $f$ consistently unless $(e,f)\in\calP$ and $\lambda^{k-1}\leq \lambda(e,f) \leq \lambda^k$.
	But, since these are two consecutive elements of $\calE$, this is the case if and only if $\lambda(e,f)\in\Set{\lambda^{k-1},\lambda^k}$.
	
	So assume that $\lambda(e,f) = \lambda^k > \lambda^{k-1}$.
	In this case, $e$ occurs before $f$ in $\calS_{\lambda^{k}}^\downarrow$.
	By \cref{spec-alg-orderings-element-swaps}, $c_{\lambda^{k-1}}(e) < c_{\lambda^{k-1}}(f)$, so $e$ occurs before $f$ in $\calS_{\lambda^{k-1}}^\uparrow$ as well.
	
	Analogously, if $\lambda(e,f) = \lambda^{k-1} < \lambda^{k}$, then $e$ occurs after $f$ in $\calS_{\lambda^{k-1}}^\uparrow$ and $c_{\lambda^{k}}(e) > c_{\lambda^{k}}(f)$, so this is the case in $\calS_{\lambda^{k-1}}^\uparrow$, too.
\end{proof}

As a next step, we show that the bases $B^k$ that \cref{alg:ESN-matroid-specific} computes satisfy $B^k = B(\calS_{\lambda^{k}}^\uparrow)$.
For this, we need one more definition concerning matroids.
\begin{definition}
	Let $M$ be a matroid, $M=(E,\calI)$, and $E'\subseteq E$.
	The \emph{restriction of $M$ to $E'$}, denoted by $\matrestr{E'}$, is the matroid $(E',\calI')$ with $\calI'\coloneqq \Set{I\in\calI : I\subseteq E'}$.
	If $E'$ is independent, then the \emph{contraction of $M$ by $E'$} is the matroid $(E\setminus E', \calI')$ with $\calI'\coloneqq\Set{I\subseteq E\setminus E' : I \cup E'\in\calI}$.
	We denote it by $\matcontr{E'}$.
\end{definition}

\begin{lemma}
	\label{spec-alg-matroid-restr-contr-and-greedy}
	Let $\calS$ be an ordering of the elements of a matroid $M$, $e\in E$, and $\hat{\calS}$ be the ordering $\calS$ without the element $e$.
	If $e\notin B(\calS)$, then $B_{\hat{M}}(\hat{\calS}) = B(\calS)$, where $\hat{M} = \matrestr{(E-e)}$.
	If $e\in B(\calS)$, then $\Set{e} \cup B_{\hat{M}}(\hat{\calS}) = B(\calS)$, where $\hat{M} = \matcontr{\Set{e}}$.
\end{lemma}
\begin{proof}
	We show that the Greedy algorithm on $M$ includes or rejects an element in $E-e$ if and only if the Greedy algorithm on $\hat{M}$ does.
	In the case that $\hat{M} = \matrestr{(E-e)}$, this is easy to see:
	since $e$ is rejected, both algorithms ask for independence on the same sets $S+e'$, yielding the same answers.
	
	If $e\in B(\calS)$, then $\Set{e}$ is independent and a set $I$ is independent in $\hat{M} = \matcontr{\Set{e}}$ if and only if $I+e$ is independent in $M$.
	Let $S^k$ and $\hat{S}^k$ be the sets after Greedy regards the $k$th element $e^k$ of~$\hat{S}$.
	Initially, $S^0=\hat{S}^0$ and let $k>0$.
	The element $e^k$ is added $\hat{S}^{k-1}$ if and only if $\hat{S}^{k-1}+e+e^k$ is independent, in which case it is also added to~$S^k$.
	Conversely, if $e^k$ is added to $S^k$, then $S^k+e^k$ is independent.
	Since $e$ is included in $B(\calS)$, it is added by the algorithm at some point and $S^k+e^k+e$ is independent as well.
	Thus, the Greedy algorithms behave the same.
\end{proof}

\begin{lemma}
	\label{spec-alg-Bk-obtained-by-Slambdak}
	Let $\lambda^1<\ldots<\lambda^s$ be the values in $\calE$ and $\lambda^0\coloneqq0$.
	For $k\in\Set{0,\ldots,s}$, the basis $B^k$ computed by \cref{alg:ESN-matroid-specific} satisfies $B^k = B(\calS_{\lambda^{k}}^\uparrow)$.
\end{lemma}
\begin{proof}
	For $B^0$ this is the case, since $\calS_0^\uparrow$ is just a lexicographic ordering by $(\objective[2]{},\objective[1]{})$.
	Assume it holds up to $B^{k-1}$, for some $k\in\Set{1,\ldots,s}$, and regard $B^k$.
	
	The operations performed by \cref{alg:ESN-matroid-specific} in iteration~$k$ are equivalent to running Greedy on the matroid $\hat{M}$ in which it all the elements in $E\setminus(B^{k-1}\cup E_{\lambda^k})$ are removed and the elements in $B^{k-1}\setminus E_{\lambda^k}$ are contracted, that is, in $\hat{M} = \matcontr[(\matrestr{(B^{k-1}\cup E_{\lambda^k})})]{(B^{k-1}\setminus E_{\lambda^k})}$.
	This is a matroid on $\hat{E} = E_{\lambda^k}$ and we use the lexicographic ordering with respect to $(c_{\lambda^k},c_1)$.
	
	By \cref{spec-alg-matroid-restr-contr-and-greedy}, if we show that $B(\calS_{\lambda^{k}}^\uparrow)$ does not contain the elements in $E\setminus(B^{k-1}\cup E_{\lambda^k})$, then we can run Greedy on $\matrestr{(B^{k-1}\cup E_{\lambda^k})}$ instead.
	If we then show that $B(\calS_{\lambda^{k}}^\uparrow)$ contains the elements in $B^{k-1}\setminus E_{\lambda^k}$, then we may contract these elements.
	With these two properties we obtain that $B^k = B(\calS_{\lambda^{k}}^\uparrow)$.
	
	Thus, we need to show that for $e\in E\setminus(B^{k-1}\cup E_{\lambda^k})$, $e\notin B(\calS_{\lambda^{k}}^\uparrow)$ and for $e\in B^{k-1}\setminus E_{\lambda^k}$, $e\in B(\calS_{\lambda^{k}}^\uparrow)$.
	This is equivalent to $B^{k-1}\setminus E_{\lambda^k} \subseteq B(\calS_{\lambda^{k}}^\uparrow)\subseteq B^{k-1}\cup E_{\lambda^k}$, which is, in turn, equivalent to
	\begin{displaymath}
		B(\calS_{\lambda^{k}}^\uparrow)\setminus E_{\lambda^k} 
		= B^{k-1}\setminus E_{\lambda^k}
		= B(\calS_{\lambda^{k-1}}^\uparrow) \setminus E_{\lambda^k}
		= B(\calS_{\lambda^{k}}^\downarrow) \setminus E_{\lambda^k}
	\end{displaymath}
	where the last two equalities use the induction hypothesis and \cref{spec-alg-orderings-for-consecutive-lambdas}.
	The orderings $\calS_{\lambda^{k}}^\uparrow$ and $\calS_{\lambda^{k}}^\downarrow$ only differ by their ordering of the sets $E_{\lambda^k}$.
	But exchanging two adjacent elements $e$ and~$e'$ in the ordering either has no effect on the result of the Greedy algorithm or it removes $e$ and adds $e'$ instead.
	Thus, the only elements, for which a different result can be obtained, are those in~$E_{\lambda^k}$.
\end{proof}

Now we know that we compute solutions $B^k=B(\calS_{\lambda^{k}}^\uparrow)$ and we just need to show that these suffice.
\begin{lemma}
	\label{spec-alg-relevant-weights-and-orderings}
	Let $y\in Y_{ESN}$, then $y = \objective{B(\calS^\uparrow_\lambda)}$ for some $\lambda\in\calE\cup\Set{0}$.
\end{lemma}
\begin{proof}
	Let $y\in\outcomeset_{ESN}$, then $y = \objective{B}$ for $B\in\feasibleset_\lambda$ and $\lambda\in(0,1)$.
	We can obtain $B$ by the Greedy algorithm using an ordering that prefers elements of $B$ if they have equal weight $c_\lambda$.
	Since exchanging elements of identical weight in the ordering has not effect on the objective, we may assume that $B = B(\calS_\lambda)$.
	
	Let $\lambda^1<\ldots<\lambda^s$ be the points in $\calE$, $\lambda_0\coloneqq 0$, $\lambda^{s+1}\coloneqq 1$, and assume that $\lambda^{k-1} < \lambda < \lambda^k$ for some $k\in\Set{1,\ldots,s+1}$.
	Then $c_\lambda(e)=c_\lambda(f)$ if and only if $c(e) = c(f)$.
	Hence, $\calS_\lambda$ is unique and, by an analogous argument to \cref{spec-alg-orderings-for-consecutive-lambdas}, $\calS_\lambda = \calS_{\lambda^{k-1}}^\uparrow$.
	Hence, $B(\calS_\lambda) = B(\calS_{\lambda^{k-1}}^\uparrow)$.
	
	Now let $\lambda=\lambda^k$ for some $k\in\Set{1,\ldots,s}$.
	We show that $\objective{B(\calS_\lambda)}$ is a convex combination of $\objective{B(\calS_\lambda^\downarrow)}$ and $\objective{B(\calS_\lambda^\uparrow)}$.
	As a result, since $y\in\outcomeset_{ESN}$, $y\in\Set{\objective{B(\calS_\lambda^\downarrow)},\objective{B(\calS_\lambda^\uparrow)}}$.
	As $\calS_\lambda^\downarrow = \calS_{\lambda^{k-1}}^\uparrow$ by \cref{spec-alg-orderings-for-consecutive-lambdas}, the claim follows.
	
	Since $B(\calS_\lambda),\, B(\calS_\lambda^\downarrow),\, B(\calS_\lambda^\uparrow)\in\feasibleset_{\lambda}$, their images $y,\, y^\downarrow,\, y^\uparrow$ all lie on the line $\Set{y' : w\T y' = w\T y}$, where $w=\left(\lambda,1-\lambda\right)\T$.
	Thus, we can write $y = \mu y^\downarrow + (1-\mu)y^\uparrow$.
	
	Since we can transform $\calS_\lambda$ into $\calS_\lambda^\downarrow$ by exchanging adjacent elements $e,\, f$ with $c_\lambda(e) = c_\lambda(f)$ and $c_1(e) \leq c_1(f)$, we see that $c_1(B(\calS_\lambda^\downarrow)) \geq c_1(B(\calS_\lambda))$.
	Analogously, $c_1(B(\calS_\lambda^\uparrow)) \leq c_1(B(\calS_\lambda))$.
	This yields $\mu\in[0,1]$ as required.
\end{proof}

This shows correctness.
\begin{theorem}
	\Cref{alg:ESN-matroid-specific} is correct, that is, it computes a set $S\subseteq\feasibleset_{SE}$ with $\objective{S} = \outcomeset_{ESN}$.
	More precisely, it actually computes a set $S\subseteq \feasibleset_{ESE}$.
\end{theorem}
\begin{proof}
	By \cref{spec-alg-relevant-weights-and-orderings}, any $y\in\outcomeset_{ESN}$ can be written as $\objective{B(\calS_\lambda^\uparrow)}$ for $\lambda\in\calE\cup\{0\}$ and by \cref{spec-alg-Bk-obtained-by-Slambdak}, we compute all these values.
	
	To see that each $B^k$ computed by the algorithm is extreme-supported, we simply realise that, by \cref{spec-alg-Bk-obtained-by-Slambdak,spec-alg-orderings-for-consecutive-lambdas}, $B^k = B(\calS_{\lambda^{k}}^\uparrow) = B(\calS_{\lambda^{k+1}}^\downarrow)$, so $\Lambda(B^k)$ contains at least the two elements $\lambda^{k} < \lambda^{k+1}$, where $\lambda^0\coloneqq 0$ and $\lambda^{s+1}\coloneqq 1$.
	Using \cref{prelims-ese-multiple-weights}, we get the desired result.
\end{proof}

\subsection{Implementing Algorithm~\ref{alg:ESN-matroid-specific}}
How do we implement \cref{alg:ESN-matroid-specific} efficiently?
Initially, we compute the intersection points $\lambda(e,f)$ for the pairs $(e,f)\in\calP$, adding both to the set $E_{\lambda(e,f)}$.
Combined with the sorting of the intersection points obtained, this takes $\O{m^2\log m}$ time.
Computing $B^0$ is again done by the Greedy algorithm and takes time $\O{m\cdot(\log m + \tindep)}$.

In the iteration for $\lambda\in\calE$, we sort a set $E_{\lambda}$ and then test independence for each of its elements.
Hence, this takes $\O{\abs{E_{\lambda}}\cdot(\log \abs{E_{\lambda}} + \tindep)}$ time.
Since each pair in $\calP$ results in at most two elements being added to a set $E_{\lambda}$, we can conclude that $\sum_{\lambda\in\calE} \abs{E_{\lambda}} \in\O{m^2}$.
Thus, these iterations take $\O{m^2\cdot(\log m + \tindep)}$ time, dominating our prior value.

\begin{theorem}
	\Cref{alg:ESN-matroid-specific} can be implemented to run in $\O{m^2\cdot(\log m + \tindep)}$ time.
\end{theorem}

We want to complete this section by providing a brief comparison to dichotomic search, which has a running time of $\Theta(\abs{\outcomeset_{ESN}}\cdot m\cdot (\log m + \tindep))$~\cite{AN79}.
From computational geometry~\cite{Dey98,Epp98}, we know that $\abs{\outcomeset_{ESN}}\in\Theta(mr^{\tfrac{1}{3}})$.
We would thus obtain $\Theta(m^2r^{\tfrac{1}{3}}\cdot (\log m + \tindep))$ as the running time for dichotomic search, which is worse than the running time of our algorithm.

This is to be taken with a grain of salt however, since the independence tests we run differ:
dichotomic search runs the Greedy algorithm while \Cref{alg:ESN-matroid-specific} removes and readds elements, which can make a real difference.
In the case of spanning trees, for example, doing the independence tests in the Greedy algorithm takes $\O{\log n}$  time~\cite{AMO93}, so it is dominated by the sorting, where the independence tests for \Cref{alg:ESN-matroid-specific} require us to check the existence of a cycle, so they need $\O{n}$ time, putting dichotomic search in the lead again.
\section{Conclusion and Future Research}
\label{sec:concl}

In this paper, we presented a generic way to obtain an alternative to dichotomic search for combinatorial problems, provided a sufficiently strong notion of adjacency.
Unlike dichotomic search, it operates in the decision space and, despite our use of weight set arguments, it does not need to solve any weighted sum scalarisation problem.
Instead, it just needs to evaluate adjacent solutions and transition to the best one.

We used this generic algorithm to solve the bi-objective minimum weight basis problem using the natural definition of adjacency.
We analysed its running time and showed that it is polynomial before developing an optimised algorithm for this special case.

Our algorithm can be extended to more objectives in the same way that dichotomic search is extended in \cite{PGE10}, which would yield an algorithm that is a hybrid between a decision and an image space method.
An interesting question is whether it can be extended to a purely decision space method, for which multi-objective Simplex algorithms could provide an orientation.

Finally, it is interesting to apply the method to more combinatorial optimisation problems and compare it computationally to dichotomic search.
In particular, it is interesting to see whether the optimised algorithm for the minimum weight basis problem is competitive with dichotomic search for a specific class of matroids.

\printbibliography

\clearpage

\section*{Funding statement}
Oliver Bachtler was funded by the Deutsche Forschungsgemeinschaft (DFG, German Research Foundation) - GRK 2982, 516090167 \enquote{Mathematics of Interdisciplinary Multiobjective Optimization}

\section*{Disclosure statement}
The authors report there are no competing interests to declare.

\end{document}